\documentclass[reqno]{amsart}
\usepackage[margin=1in]{geometry}
\usepackage[utf8]{inputenc}
\usepackage[T1]{fontenc}
\usepackage{graphicx}
\usepackage{xcolor}
\usepackage{mathrsfs}
\usepackage{amsthm}
\usepackage{amsmath}
\allowdisplaybreaks
\usepackage{amsfonts}
\usepackage{amssymb}
\usepackage{mathtools}
\usepackage{enumitem}
\usepackage[hidelinks]{hyperref}
\usepackage{tikz-cd}
\usepackage{tikz}
\usetikzlibrary{tqft}
\usetikzlibrary{positioning}
\usetikzlibrary{matrix,quotes,arrows.meta,bending}
\usetikzlibrary{decorations.markings}
\usetikzlibrary{calc}

\numberwithin{equation}{section}

\theoremstyle{plain}
\newtheorem{Thm}{Theorem}[section]
\newtheorem{Prop}[Thm]{Proposition}
\newtheorem{Lem}[Thm]{Lemma}

\theoremstyle{definition}
\newtheorem{Def}[Thm]{Definition}

\newcommand{\defn}[1]{\textbf{\textit{#1}}}

\newcommand{\R}{\mathbb{R}}
\newcommand{\G}{\mathcal{G}}
\newcommand{\g}{\mathfrak{g}}
\newcommand{\h}{\mathfrak{h}}
\newcommand{\too}{\longrightarrow}

\newcommand{\mtoo}{\longmapsto}
\newcommand{\Ad}{\operatorname{Ad}}

\newcommand{\M}{\mathcal{M}}

\newcommand{\A}{\mathcal{A}}
\newcommand{\K}{\mathcal{K}}
\newcommand{\B}{\mathcal{B}}

\newcommand{\im}{\operatorname{im}}
\newcommand{\ad}{\operatorname{ad}}
\newcommand{\sll}[1]{\mkern-4mu\mathbin{/\mkern-5mu/}_{\mkern-4mu{#1}}}
\newcommand{\Lie}{\operatorname{Lie}}
\newcommand{\sint}{_{\mathrm{int}}}
\newcommand{\seg}{%
\tikz[baseline=-0.5ex]{
  \fill (0,0) circle (1.5pt);
  \draw[line width=0.8pt] (0,0) -- (0.35,0);
  \fill (0.35,0) circle (1.5pt);
}}
\newcommand{\tqftcap}{\begin{tikzpicture}[
		baseline=-2.5pt,
		every tqft/.append style={
			transform shape, rotate=90, tqft/circle x radius=4pt,
			tqft/circle y radius= 2pt,
			tqft/boundary separation=0.6cm, 
			tqft/view from=incoming,
		}
		]
		\pic[
		tqft/cap,
		name=d,
		every incoming lower boundary component/.style={draw},
		every outgoing lower boundary component/.style={draw},
		every incoming boundary component/.style={draw},
		every outgoing boundary component/.style={draw},
		cobordism edge/.style={draw},
		cobordism height= 1cm,
		];
	\end{tikzpicture}}
\newcommand{\tqftcup}{\begin{tikzpicture}[
		baseline=-2.5pt,
		every tqft/.append style={
			transform shape, rotate=90, tqft/circle x radius=4pt,
			tqft/circle y radius= 2pt,
			tqft/boundary separation=0.6cm, 
			tqft/view from=incoming,
		}
		]
		\pic[
		tqft/cup,
		name=d,
		every incoming lower boundary component/.style={draw},
		every outgoing lower boundary component/.style={draw},
		every incoming boundary component/.style={draw},
		every outgoing boundary component/.style={draw},
		cobordism edge/.style={draw},
		cobordism height= 1cm,
		];
	\end{tikzpicture}}
\newcommand{\tqftpoptwoone}{%
\begin{tikzpicture}[
  baseline=2.5pt,
  every tqft/.append style={
    transform shape, rotate=90,
    tqft/circle x radius=2pt,
    tqft/circle y radius=1pt,
    tqft/boundary separation=0.3cm,
    tqft/view from=incoming,
  }
]
  \pic[
    tqft/reverse pair of pants, 
    name=d,
    every incoming lower boundary component/.style={draw},
    every outgoing lower boundary component/.style={draw},
    every incoming boundary component/.style={draw},
    every outgoing boundary component/.style={draw},
    cobordism edge/.style={draw},
    cobordism height=0.3cm,
  ];
\end{tikzpicture}%
}
\newcommand{\tqftpoponetwo}{%
\begin{tikzpicture}[
  baseline=-2.5pt,
  every tqft/.append style={
    transform shape, rotate=90,
    tqft/circle x radius=2pt,
    tqft/circle y radius=1pt,
    tqft/boundary separation=0.3cm,
    tqft/view from=incoming,
  }
]
  \pic[
    tqft/pair of pants, 
    name=d,
    every incoming lower boundary component/.style={draw},
    every outgoing lower boundary component/.style={draw},
    every incoming boundary component/.style={draw},
    every outgoing boundary component/.style={draw},
    cobordism edge/.style={draw},
    cobordism height=0.3cm,
  ];
\end{tikzpicture}%
}
\newcommand{\tqftcyl}{%
\begin{tikzpicture}[
  baseline=-2.5pt,
  every tqft/.append style={
    transform shape, rotate=90,
    tqft/circle x radius=3pt,
    tqft/circle y radius=1.5pt,
    tqft/boundary separation=0.3cm,
    tqft/view from=incoming,
  }
]
  \pic[
    tqft/cylinder,
    name=d,
    every incoming lower boundary component/.style={draw},
    every outgoing lower boundary component/.style={draw},
    every incoming boundary component/.style={draw},
    every outgoing boundary component/.style={draw},
    cobordism edge/.style={draw},
    cobordism height=0.5cm,
  ];
\end{tikzpicture}%
}
\newcommand{\tqftswap}{%
\begin{tikzpicture}[
  baseline=0.25pt,
  every tqft/.append style={
    transform shape, rotate=90,
    tqft/circle x radius=2pt,
    tqft/circle y radius=1pt,
    tqft/boundary separation=0.4cm,
    tqft/view from=incoming,
  }
]
  \pic[
    tqft/cylinder to next,
    name=d,
    every incoming lower boundary component/.style={draw},
    every outgoing lower boundary component/.style={draw},
    every incoming boundary component/.style={draw},
    every outgoing boundary component/.style={draw},
    cobordism edge/.style={draw},
    cobordism height=0.4cm,
  ];
  \pic[
    tqft/cylinder to prior,
    every incoming lower boundary component/.style={draw},
    every outgoing lower boundary component/.style={draw},
    every incoming boundary component/.style={draw},
    every outgoing boundary component/.style={draw},
    cobordism edge/.style={draw},
    cobordism height=0.4cm,
    at={($(d-incoming boundary)+(0,0.2)$)} 
  ];
\end{tikzpicture}%
}

\title{Lax--Kirchhoff moduli spaces and Hamiltonian 2D TQFT}
\date{\today}
\author{Mohamed Moussadek Maiza}
\author{Maxence Mayrand}

\begin{document}

\begin{abstract}
We introduce the \emph{Lax--Kirchhoff moduli space} associated with a finite quiver $\Gamma$ and a compact connected Lie group $G$. On each oriented edge we consider the Lax equation $\dot{A}_1 + [A_0, A_1] = 0$ and impose a Kirchhoff-type matching condition for the fields $A_1$ at interior vertices. Modulo gauge transformations trivial on the boundary, this yields a moduli space $\M(\Gamma)$. We prove that $\M(\Gamma)$ is a finite-dimensional smooth symplectic manifold carrying a Hamiltonian action of $G^{\partial\Gamma}$ whose moment map records the boundary values of $A_1$. Analytically, we construct slices for the infinite-dimensional gauge action and realize $\M(\Gamma)$ by Marsden--Weinstein reduction. For the quiver consisting of a single edge, we recover the classical identification $\M \cong T^*G$. In general, we identify $\M(\Gamma)$ with a symplectic reduction of $T^*G^E$ by $G^{\Gamma_{\mathrm{int}}}$, where $E$ is the set of edges and $\Gamma_{\mathrm{int}}$ is the set of interior vertices. We further show that $\M(\Gamma)$ is invariant under quiver homotopies, implying that it depends only on the surface with boundary obtained by thickening $\Gamma$. We then assemble these spaces into a two-dimensional topological quantum field theory valued in a category of Hamiltonian spaces.
\end{abstract}

\maketitle

\section{Introduction}

Let $G$ be a Lie group with Lie algebra $\g$. 
The \emph{Lax equation} is the ordinary differential equation
\[
\dot{A}_1 + [A_0, A_1] = 0
\]
for pairs of elements $A_0, A_1$ in $\g$ depending on a real variable $t$. 
This equation plays a central role in many areas of mathematics, most prominently in the theory of integrable systems, where it generates isospectral flows and hence an abundance of conserved quantities \cite{Lax1968Integrals,Perelomov1990,Fomenko1988,FaddeevTakhtajan1987,HitchinSegalWard2013Integrable}. 
It also appears naturally in gauge theory---such as the $(1+1)$-dimensional Yang--Mills equations on a spacetime cylinder \cite{DriverHall1999YangMills,Hall2002Geometric} or as the complex part of Nahm’s equations \cite{Donaldson1984Nahm,kronheimer1988hyperkahler,Kronheimer1990coadjoint,Kovalev1996Nahm,Biquard1996Nahm}---as well as in Poisson geometry, where it encodes, for instance, the Gauss law on $\g^*$ \cite{CattaneoFelder2001Poisson}.

Geometrically, the Lax equation can be viewed as the condition that a $\g$-valued function $A_1$ be parallel with respect to the connection $A_0dt$ on the trivial principal $G$-bundle over the interval. 
Equivalently, it is the zero-curvature condition for an $S^1$-invariant connection $A_0dt + A_1d\theta$ on the cylinder $[0,1] \times S^1$. 
This formulation naturally gives rise to an action of the group of \emph{gauge transformations}, i.e.\ maps $g : [0,1] \to G$, acting by\footnote{We use matrix notation for convenience; the expression make sense for any Lie group via the adjoint and translation actions.}
\[
g \cdot (A_0, A_1) = (gA_0g^{-1} - \dot{g}g^{-1}, gA_1g^{-1}).
\]
If we restrict to continuously differentiable solutions $A = (A_0, A_1)$ to the Lax equation and quotient by the subgroup $\G_0$ of twice continuously differentiable gauge transformations satisfying $g(0) = g(1) = 1$, we obtain a moduli space
\[
\M([0, 1]) \coloneqq \{A \in C^1([0,1], \g \times \g) : \dot{A}_1 + [A_0, A_1] = 0\} / \G_0
\]
of solutions to the Lax equation on the interval $[0,1]$. 
In particular, when $G$ is a compact connected Lie group, $\M([0,1])$ is a finite-dimensional smooth symplectic manifold isomorphic to $T^*G$ \cite{Hall2002Geometric,kronheimer1988hyperkahler,DancerSwann1996Hyperkahler,Bielawski2007Lie}. 
Moreover, the action of the full gauge group $\G \coloneqq C^2([0,1], G)$ with free boundary values descends to a Hamiltonian action of $\G / \G_0 \cong G \times G$, with moment map
\[
\M([0, 1]) \too \g \times \g, \quad (A_0, A_1) \mtoo (-A_1(0), A_1(1)).
\]
Under the identification $\M([0, 1]) \cong T^*G$, this is precisely the standard Hamiltonian action of $G \times G$ on $T^*G$ induced by left and right multiplication on $G$.

We generalize the Lax moduli space $\M([0, 1])$ by replacing the interval $[0,1]$ with a quiver $\Gamma = (V, E, s, t)$ and imposing a Kirchhoff-type condition for $A_1$ at the interior vertices. 
In more detail, the vertex set of $\Gamma$ naturally decomposes as
\[
V = \partial\Gamma^- \sqcup \partial\Gamma^+ \sqcup \Gamma_\mathrm{int},
\]
where $\partial\Gamma^-$ consists of the degree-$1$ vertices with an incoming edge, $\partial\Gamma^+$ of those with an outgoing edge, and $\Gamma_\mathrm{int}$, the \emph{interior vertices}, those that have degree greater than $1$ (we assume there are no isolated vertices).
\begin{equation}\label{er7w843z}
\begin{tikzcd}[
  sep=small,
  column sep={2cm,between origins},
  row sep={0.5cm,between origins},
  cells={nodes={inner sep=0pt, outer sep=0pt}},
  chevW/.store in=\chevW, chevW=0.10cm,
  chevH/.store in=\chevH, chevH=0.07cm,
  chevShort/.store in=\chevShort, chevShort=0.1ex,
  chevarrow/.style={
    no head,
    line width=0.9pt, line cap=round,
    shorten <=-\chevShort, shorten >=-\chevShort,
    postaction={
      decorate,
      decoration={
        markings,
        mark=at position .5 with {
          \draw[line width=0.9pt] (-\chevW,-\chevH) -- (0,0) -- (-\chevW,\chevH);
        }
      }
    }
  }
]
    & \textcolor{red}{\bullet}  \\
    & \\
    & {\text{\textcolor{red}{$\substack{\text{incoming}\\\text{boundaries}\\\partial\Gamma^-}$}}}
      & \bullet
      & \textcolor{blue}{\bullet} & \\
    &  & \substack{\text{interior}\\\text{vertex}\\\Gamma\sint}
      & {\text{\textcolor{blue}{$\substack{\text{outgoing}\\\text{boundary}\\\partial\Gamma^+}$}}} \\
    & \textcolor{red}{\bullet}
    \arrow[from=1-2, to=3-3, chevarrow, shorten <= -0.5pt]
    \arrow[from=3-3, to=3-4, chevarrow, shorten >= -0.5pt]
    \arrow[from=5-2, to=3-3, chevarrow, shorten <= -0.5pt]
\end{tikzcd}
\end{equation}
Each edge represents a copy of $[0,1]$, and on each edge $e \in E$ we consider $C^1$ solutions $A^e = (A_0^e, A_1^e)$ to the Lax equation, subject to the Kirchhoff law \eqref{onnglawg} for $A_1^e$ at the interior vertices (see \S\ref{9qzubmd9} for details). 
The gauge group $\G_0(\Gamma)$ consists of tuples $g = (g_e)_{e \in E}$ with $g_e \in C^2([0,1], G)$, matching their boundary values at the interior vertices and satisfying trivial boundary conditions on $\partial\Gamma^\pm$. 
Its action on the space $\A(\Gamma)$ of edgewise solutions to the Lax equation satisfying the Kirchhoff law gives rise to the moduli space
\[
\M(\Gamma) \coloneqq \A(\Gamma) / \G_0(\Gamma)
\]
of solutions to the Lax--Kirchhoff equations on $\Gamma$. 
The action of the full gauge group $\G(\Gamma)$, with matching boundary values at interior vertices and free boundary values on $\partial\Gamma^\pm$, descends to an action of
\[
\G(\Gamma)/\G_0(\Gamma) \cong G^{\partial\Gamma}.
\]
We give a rigorous construction of a smooth manifold structure and symplectic form on $\M(\Gamma)$ via Marsden--Weinstein reduction \cite{MarsdenWeinstein}.

\begin{Thm}
Let $\Gamma = (V, E, s, t)$ be a connected quiver with non-empty boundary. 
Then $\M(\Gamma)$ is a finite-dimensional smooth symplectic manifold of dimension
\[
\dim \M(\Gamma) = 2(|E| - |\Gamma\sint|)\dim G
\]
and the action of $G^{\partial\Gamma}$ on $\M(\Gamma)$ is Hamiltonian with moment map
\[
\M(\Gamma) \too \g^{\partial\Gamma}, 
\quad
A \mtoo \bigl(\operatorname{sgn}(v) A_1(v)\bigr)_{v \in \partial\Gamma}.
\]
Moreover, there is an isomorphism of Hamiltonian $G^{\partial\Gamma}$-spaces
\[
\M(\Gamma) \cong T^*G^E \sll{} G^{\Gamma_\mathrm{int}},
\]
where $G^{\Gamma_\mathrm{int}} \subset G^V$ acts on $T^*G^E$ via the embedding $G^V \to (G \times G)^E$, $(g_v)_{v \in V} \mapsto (g_{t(e)}, g_{s(e)})_{e \in E}$.
\end{Thm}

The key ingredient in the proof is an explicit construction of local slices for the action of $\G_0(\Gamma)$ on the Banach manifold $\A(\Gamma)$.

Next, we study how the moduli spaces $\M(\Gamma)$ behave under elementary operations on quivers. 
In particular, given two quivers $\Gamma_1$ and $\Gamma_2$ with a common boundary set $\partial \Gamma_1^+ = \partial \Gamma_2^- \eqqcolon B$, we form the quiver $\Gamma_1 \star \Gamma_2$ by gluing $\Gamma_1$ and $\Gamma_2$ along $B$. 
\begin{equation}\label{7w8c1eku}
\begin{tikzpicture}[
  baseline=(current bounding box.center),
  x=0.8cm, y=-0.8cm, 
  vertex/.style={circle, fill=black, draw=none, inner sep=1.2pt, outer sep=0pt},
  edge/.style={line width=0.9pt, line cap=round},
  chevW/.store in=\chevW, chevW=0.08, 
  chevH/.store in=\chevH, chevH=0.08,
  markchev/.style={
    postaction={
      decorate,
      decoration={
        markings,
        mark=at position #1 with {
          \draw[edge] (-\chevW,-\chevH) -- (0,0) -- (-\chevW,\chevH);
        }
      }
    }
  },
  markchev/.default=0.55
]
  \coordinate (V11) at (0,0); 
  \coordinate (V31) at (0,2); 
  \coordinate (V22) at (1,1); 
  \coordinate (V23) at (2,1); 
  \coordinate (V14) at (3,0); 
  \coordinate (V34) at (3,2); 

  \draw[edge, markchev] (V11) -- (V22); 
  \draw[edge, markchev] (V31) -- (V22); 
  \draw[edge, markchev] (V22) -- (V23); 
  \draw[edge, markchev] (V23) -- (V14); 
  \draw[edge, markchev] (V23) -- (V34); 

  \foreach \v in {V11,V31,V22,V23,V14,V34} {
    \node[vertex] at (\v) {};
  }
\end{tikzpicture}
\qquad\star\qquad
\begin{tikzpicture}[
  baseline=(current bounding box.center),
  x=0.8cm, y=-0.8cm, 
  vertex/.style={circle, fill=black, draw=none, inner sep=1.2pt, outer sep=0pt},
  edge/.style={line width=0.9pt, line cap=round},
  chevW/.store in=\chevW, chevW=0.1, 
  chevH/.store in=\chevH, chevH=0.1,
  markchev/.style={
    postaction={
      decorate,
      decoration={
        markings,
        mark=at position #1 with {
          \draw[edge] (-\chevW,-\chevH) -- (0,0) -- (-\chevW,\chevH);
        }
      }
    }
  },
  markchev/.default=0.55
]
  \coordinate (V00) at (0,0); 
  \coordinate (V20) at (0,2); 
  \coordinate (V11) at (1,1); 
  \coordinate (V12) at (2,1); 

  \draw[edge, markchev] (V00) -- (V11); 
  \draw[edge, markchev] (V20) -- (V11); 
  \draw[edge, markchev] (V11) -- (V12); 

  \foreach \v in {V00,V20,V11,V12} {
    \node[vertex] at (\v) {};
  }
\end{tikzpicture}
\qquad=\qquad
\begin{tikzpicture}[
  baseline=(current bounding box.center),
  x=0.8cm, y=-0.8cm, 
  vertex/.style={circle, fill=black, draw=none, inner sep=1.2pt, outer sep=0pt},
  edge/.style={line width=0.9pt, line cap=round},
  chevW/.store in=\chevW, chevW=0.08, 
  chevH/.store in=\chevH, chevH=0.08,
  markchev/.style={
    postaction={
      decorate,
      decoration={
        markings,
        mark=at position #1 with {
          \draw[edge] (-\chevW,-\chevH) -- (0,0) -- (-\chevW,\chevH);
        }
      }
    }
  },
  markchev/.default=0.55
]
  \coordinate (V11) at (0,0); 
  \coordinate (V31) at (0,2); 
  \coordinate (V22) at (1,1); 
  \coordinate (V23) at (2,1); 
  \coordinate (V14) at (3,0); 
  \coordinate (V34) at (3,2); 
  \coordinate (V25) at (4,1); 
  \coordinate (V26) at (5,1); 

  \draw[edge, markchev] (V11) -- (V22); 
  \draw[edge, markchev] (V31) -- (V22); 
  \draw[edge, markchev] (V22) -- (V23); 
  \draw[edge, markchev] (V23) -- (V14); 
  \draw[edge, markchev] (V23) -- (V34); 
  \draw[edge, markchev] (V14) -- (V25); 
  \draw[edge, markchev] (V34) -- (V25); 
  \draw[edge, markchev] (V25) -- (V26); 

  \foreach \v in {V11,V31,V22,V23,V14,V34,V25,V26} {
    \node[vertex] at (\v) {};
  }
\end{tikzpicture}
\end{equation}
This operation is reflected on moduli by Hamiltonian reduction:

\begin{Thm}
If $\partial \Gamma_1^+ = \partial \Gamma_2^- \eqqcolon B$, then there is a canonical isomorphism
\[
\M(\Gamma_1 \star \Gamma_2) \cong \bigl(\M(\Gamma_1) \times \M(\Gamma_2)\bigr) \sll{} G^B,
\]
as Hamiltonian $G^{\partial\Gamma_1^-} \times G^{\partial\Gamma_2^+}$-spaces.
\end{Thm}

To a quiver $\Gamma$ we associate an oriented $2$-dimensional cobordism $\Sigma_\Gamma$ (a surface with boundary) by ``thickening'' it as in the following picture:
\begin{equation}\label{fyqqebnn}
\begin{tikzpicture}[
  baseline=(current bounding box.center),
  x=0.004cm, y=-0.004cm,
  vertex/.style={circle, fill=black, draw=none, inner sep=1.2pt, outer sep=0pt},
  edge/.style={line width=0.9pt, line cap=round},
  chevW/.store in=\chevW, chevW=16,
  chevH/.store in=\chevH, chevH=16,
  markchev/.style={
    postaction={
      decorate,
      decoration={
        markings,
        mark=at position #1 with {
          \draw[edge] (-\chevW,-\chevH) -- (0,0) -- (-\chevW,\chevH);
        }
      }
    }
  },
  markchev/.default=0.55
]
  \coordinate (A) at (409,170);
  \coordinate (B) at (1268,169);
  \coordinate (C) at (134,441);
  \coordinate (D) at (244,898);
  \coordinate (E) at (1661,536);
  \coordinate (F) at (1485,918);
  \coordinate (G) at (345,633);
  \coordinate (H) at (643,688);
  \coordinate (I) at (778,430);
  \coordinate (J) at (535,408);
  \coordinate (K) at (1134,428);
  \coordinate (L) at (1123,667);
  \coordinate (M) at (1390,629);
  \coordinate (N) at (1271,516);

  \draw[edge, markchev] (C) -- (G);
  \draw[edge, markchev] (D) -- (G);
  \draw[edge, markchev] (H) -- (G);
  \draw[edge, markchev] (G) -- (J);
  \draw[edge, markchev] (J) -- (A);
  \draw[edge, markchev] (J) -- (I);
  \draw[edge, markchev] (I) -- (H);
  \draw[edge, markchev] (I) -- (K);
  \draw[edge, markchev] (K) -- (B);
  \draw[edge, markchev] (K) -- (N);
  \draw[edge, markchev] (N) -- (M);
  \draw[edge, markchev] (K) -- (L);
  \draw[edge, markchev] (L) -- (M);
  \draw[edge, markchev] (M) -- (E);
  \draw[edge, markchev] (F) -- (M);

  \foreach \v in {A,B,C,D,E,F,G,H,I,J,K,L,M,N} {
    \node[vertex] at (\v) {};
  }
\end{tikzpicture}
\quad\mtoo\quad
\raisebox{-0.51\height}{\includegraphics[width=2.9in]{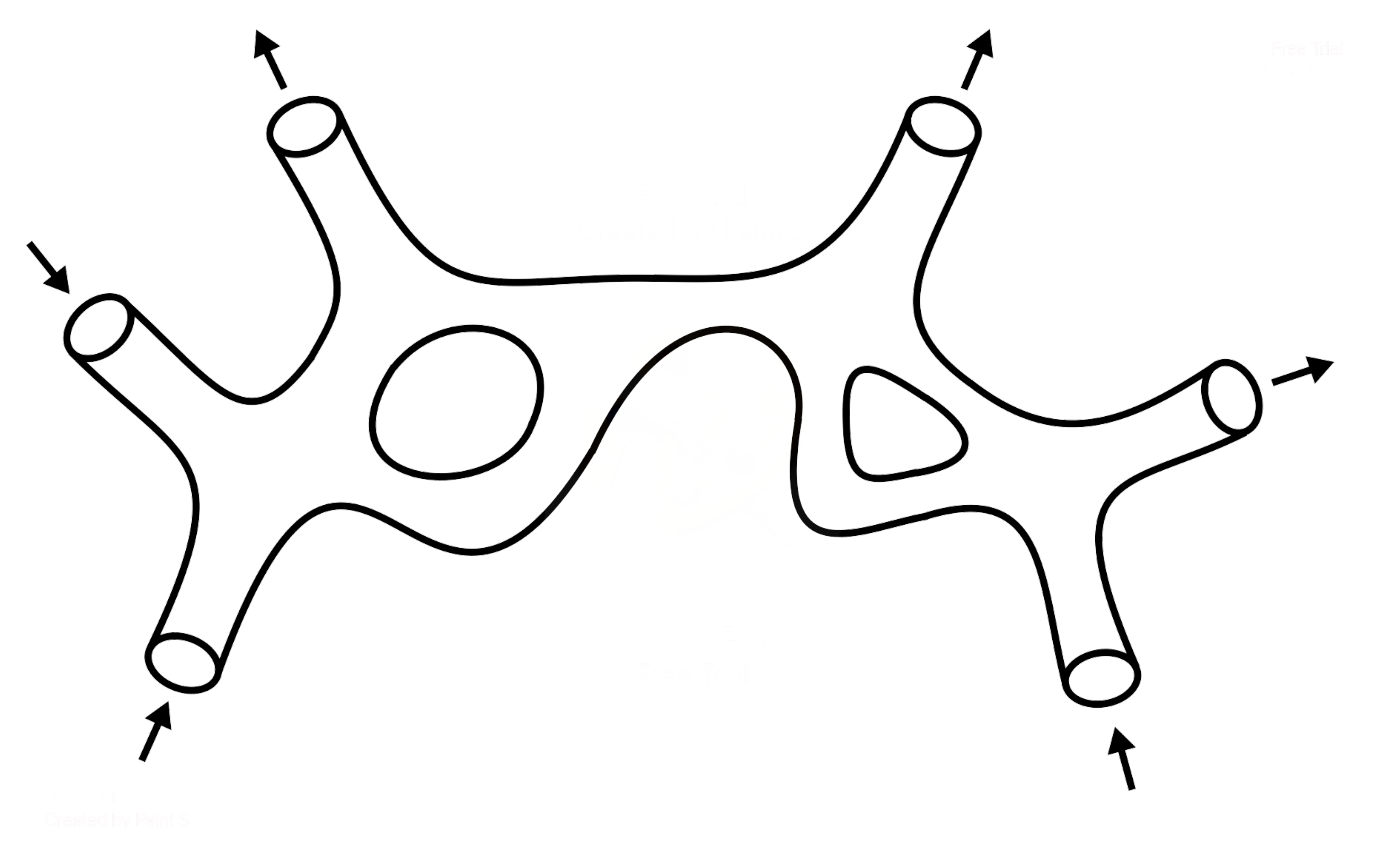}}
\end{equation}
\[
\kern-1.5cm
\begin{tikzcd}
\Gamma  & \mtoo & \Sigma_\Gamma
\end{tikzcd}
\]
Different quivers may induce isomorphic cobordisms; the precise combinatorial relationship between quivers with isomorphic associated cobordisms is given by \emph{quiver homotopy}; see Definition \ref{nyfvbnh4} and Proposition \ref{pahn7k6y}.

\begin{Thm}
If $\Gamma_1$ and $\Gamma_2$ are homotopic, then $\M(\Gamma_1) \cong \M(\Gamma_2)$ as Hamiltonian spaces. 
In particular, the isomorphism class of $\M(\Gamma)$ depends only on the underlying cobordism $\Sigma_\Gamma$.
\end{Thm}

Finally, we use these results to construct a $2$-dimensional topological quantum field theory (TQFT) valued in a category of Hamiltonian spaces, in the spirit of the Moore--Tachikawa TQFT for holomorphic symplectic manifolds \cite{MooreTachikawa}, but in the real setting (see also \cite{Cazassus2023twocategory,CrooksMayrand2024MooreTachikawa}). 
Namely, define a category $\mathbf{Ham}$ whose objects are Lie groups and where a morphism $G \to H$ is a Hamiltonian $G \times H$-space. 
Composition of morphisms $M : G \to H$ and $N : H \to I$ is given by Hamiltonian reduction of $M \times N$ by $H$ (this is only a partial category but it can be completed into a category as we explain in \S\ref{j8sih92i}).
Then assigning $S^1 \mapsto G$ and $\Sigma_\Gamma \mapsto \M(\Gamma)$ yields a symmetric monoidal functor
\[
\M : \mathbf{Cob}_2 \too \mathbf{Ham},
\]
i.e.\ a TQFT with values in $\mathbf{Ham}$.

\subsection*{Organization}
Section \ref{9qzubmd9} sets the precise definition of the Lax--Kirchhoff moduli space $\M(\Gamma)$. We then prove in Section \ref{zscudduj} that $\M(\Gamma)$ is a finite-dimensional smooth manifold by constructing local slices for the action. Section \ref{vpbytyos} then uses the infinite-dimensional version of Marsden--Weinstein reduction to endow $\M(\Gamma)$ with a symplectic structure. It is then shown in Section \ref{rj2bih75} that $\M(\Gamma)$ is a Hamiltonian $G^{\partial\Gamma}$-space via an explicit identification with a symplectic reduction of $T^*G^E$ by $G^{\Gamma\sint}$. We then prove the homotopy invariance of $\M(\Gamma)$ in Section \ref{x63cjnor} and use it in Section \ref{j8sih92i} to construct a TQFT valued in $\mathbf{Ham}$.

\subsection*{Acknowledgments} We thank Anton Alekseev and Arturo Zenen Oliva Gonzalez for useful discussions. The second author acknowledges the support of a Discovery Grant (RGPIN-2023-04587) from the Natural Sciences and Engineering Research Council of Canada (NSERC). The first author thanks the Institut des sciences mathématiques ISM and the Ernest-Monga scholarship for their financial support.

\section{The Lax--Kirchhoff moduli space}\label{9qzubmd9}

The goal of this section is to give a precise definition of the Lax--Kirchhoff moduli space $\M(\Gamma)$. 
Throughout this section and the rest of the paper, $G$ denotes a compact connected Lie group with Lie algebra $\g$ endowed with an $\Ad$-invariant inner product. We also work in a representation of $G$ as a matrix Lie group so that, for example, $-\dot{g}g^{-1}$ is well-defined for $g \in C^2([0, 1], G)$. Although this last step is not necessary, it simplifies the notation and the proofs substantially.

\begin{Def}
A \defn{quiver} is a tuple $\Gamma = (V, E, s, t)$ where:
\begin{itemize}
    \item $V$ is a finite set of \defn{vertices};
    \item $E$ is a finite set of \defn{edges};
    \item $s, t : E \to V$ are maps called the \defn{source} and \defn{target} maps, respectively.
\end{itemize}
We assume that $\Gamma$ has no isolated vertices, i.e.\ $\operatorname{im}(s) \cup \operatorname{im}(t) = V$.  
The \defn{degree} of a vertex $v \in V$ is the sum
\[
\deg(v) \coloneqq \deg_{\mathrm{in}}(v) + \deg_{\mathrm{out}}(v),
\]
where $\deg_{\mathrm{in}}(v) = |t^{-1}(v)|$ is the number of incoming edges and $\deg_{\mathrm{out}}(v) = |s^{-1}(v)|$ the number of outgoing ones.
The \defn{boundary} of $\Gamma$, denoted $\partial\Gamma$, is the set of vertices of degree $1$. 
It decomposes as a disjoint union
\[
\partial\Gamma = \partial\Gamma^- \sqcup \partial\Gamma^+,
\]
where
\[
\partial\Gamma^- \coloneqq \{v \in \partial\Gamma : \deg_{\mathrm{out}}(v) = 1\}
\quad\text{and}\quad
\partial\Gamma^+ \coloneqq \{v \in \partial\Gamma : \deg_{\mathrm{in}}(v) = 1\}
\]
are the \defn{incoming} and \defn{outgoing boundaries}, respectively.  
The complement
\[
\Gamma\sint \coloneqq V \setminus \partial\Gamma
\]
is the set of \defn{interior vertices}. See Figure \eqref{er7w843z}.
\end{Def}

Let $\Gamma = (V, E, s, t)$ be a quiver.  
To each edge $e \in E$ we associate a copy $I_e$ of the interval $[0,1]$, and let 
\[
I \coloneqq \bigsqcup_{e \in E} I_e.
\]
We consider the Banach space
\[
\B(\Gamma) \coloneqq C^1(I, \g \times \g) = \bigl(C^1([0,1], \g \times \g)\bigr)^E,
\]
whose elements are tuples $A = (A^e)_{e \in E}$, where each $A^e = (A_0^e, A_1^e)$ satisfies $A_i^e : I_e \to \g$ for $i = 0,1$.  
We also write $A_i \coloneqq (A_i^e)_{e \in E}$ for the components of $A$.

\medskip

Inside $\B(\Gamma)$, we define the closed subspace called the \defn{Kirchhoff space}
\begin{equation}\label{onnglawg}
\K(\Gamma) \coloneqq 
\left\{
A \in \B(\Gamma) : 
\sum_{e \in t^{-1}(v)} A_1^e(1) = \sum_{e \in s^{-1}(v)} A_1^e(0)
\text{ for all } v \in \Gamma\sint
\right\}.
\end{equation}
This expresses the Kirchhoff law for the $\g$-valued functions $A_1^e$ at each interior vertex.  

Next, define the \defn{Lax space}
\[
\mathcal{L}(\Gamma) \coloneqq 
\{ A \in \B(\Gamma) : \dot{A}_1 + [A_0, A_1] = 0 \},
\]
and the \defn{Lax--Kirchhoff space}
\[
\A(\Gamma) \coloneqq \mathcal{L}(\Gamma) \cap \mathcal{K}(\Gamma).
\]

\medskip

The Banach space $\B(\Gamma)$ carries a natural action of the Banach Lie group 
\[
C^2(I, G) = \bigl(C^2([0,1], G)\bigr)^E,
\]
by the gauge transformations
\[
g \cdot (A_0, A_1) = (gA_0g^{-1} - \dot{g}g^{-1}, gA_1g^{-1}),
\qquad g \in C^2(I, G),\; A \in \B(\Gamma).
\]
Let $\G(\Gamma)$ be the closed subgroup of $C^2(I, G)$ consisting of elements that match at the vertices; that is, an element $g = (g_e)_{e \in E} \in C^2(I, G)$ lies in $\G(\Gamma)$ if and only if there exist $g(v) \in G$ for each $v \in V$ such that $g_e(1) = g(t(e))$ and $g_e(0) = g(s(e))$ for all $e \in E$.
Then $\G(\Gamma)$ preserves the Lax--Kirchhoff space $\A(\Gamma)$.  
We denote by $\G_0(\Gamma)$ its normal subgroup consisting of elements that are trivial at the boundary, i.e.\ 
\[
\G_0(\Gamma) \coloneqq \{ g \in \G(\Gamma) : g(v) = 1 \text{ for all } v \in \partial \Gamma\}
\]

\begin{Def}
The \defn{Lax--Kirchhoff moduli space} of a quiver $\Gamma$ is the quotient
\[
\M(\Gamma) \coloneqq \A(\Gamma) / \G_0(\Gamma).
\]
\end{Def}

Finally, note that the action of $\G(\Gamma)$ on $\A(\Gamma)$ descends to an action of the quotient group
\[
\G(\Gamma) / \G_0(\Gamma) \cong G^{\partial\Gamma}
\]
on $\M(\Gamma)$.

\section{The smooth structure}\label{zscudduj}

Let $\Gamma = (V, E, s, t)$ be a connected quiver with non-empty boundary. 
In this section we show that the Lax--Kirchhoff moduli space $\M(\Gamma)$ carries a natural finite-dimensional smooth manifold structure.

\begin{Thm}[Smooth structure]\label{pa1025dm}
The subset $\A(\Gamma)$ is a Banach submanifold of $\B(\Gamma)$, and the quotient
\[
\M(\Gamma) \coloneqq \A(\Gamma) / \G_0(\Gamma)
\]
admits a unique smooth manifold structure for which the projection $\A(\Gamma) \to \M(\Gamma)$ is a smooth submersion. 
Moreover, $\M(\Gamma)$ is finite-dimensional with
\[
\dim \M(\Gamma) = 2(|E| - |\Gamma\sint|)\dim G.
\]
\end{Thm}

The proof proceeds in three steps. 
First, we verify that $\A(\Gamma)$ is a Banach submanifold of $\B(\Gamma)$. 
Second, we show that the $\G_0(\Gamma)$-action on $\A(\Gamma)$ is free and proper. 
Finally, as the infinite-dimensional setting requires, we construct local slices for this action. 
For general background on Banach manifolds, see, for example, \cite{Lang1999Fundamentals}.

\subsection{The Banach submanifold $\A(\Gamma) \subset \B(\Gamma)$}

Consider the map
\begin{equation}\label{bbqchqb6}
\psi : \B(\Gamma) \too C^0(I, \g) \times \g^{\Gamma\sint}, 
\quad 
\psi(A) = 
\left(
\dot{A}_1 + [A_0, A_1],
\left(
\sum_{e \in s^{-1}(v)} A_1^e(0) - \sum_{e \in t^{-1}(v)} A_1^e(1)
\right)_{v \in \Gamma\sint}
\right),
\end{equation}
so that $\A(\Gamma) = \psi^{-1}(0)$. 
In the next section we will interpret $\psi$ as a moment map for the action of $\G_0(\Gamma)$ on $\B(\Gamma)$. 
To show that $0$ is a regular value of $\psi$, we use the following general lemma, which will also be useful in other parts of the paper.

\begin{Lem}\label{qanw51t3}
Let $W$ be a finite-dimensional real vector space and $B \in C^0(I, \operatorname{End}(W))$. 
Define
\begin{align*}
R : C^1(I, W) &\too C^0(I, W) \times W^{\Gamma\sint},\\
x &\mtoo 
\left(
\dot{x} + Bx,
\left(
\sum_{e \in s^{-1}(v)} x_e(0) - \sum_{e \in t^{-1}(v)} x_e(1)
\right)_{v \in \Gamma\sint}
\right).
\end{align*}
Then $R$ is surjective and 
\[
\dim \ker R = (|E| - |\Gamma\sint|)\dim W.
\]
Moreover, if $T \subset E$ is a spanning tree of the underlying undirected graph, rooted at $r_0 \in \partial\Gamma^-$ with $r_0 = s(e_0)$, then for each $(a, b) \in C^0(I, W) \times W^{\Gamma\sint}$, the map
\begin{align}
R^{-1}(a, b) &\too W^{E \setminus T} \times W^{s^{-1}(\partial\Gamma) \setminus \{e_0\}} \times W^{t^{-1}(\partial\Gamma)}, \label{di992oem}\\
x &\mtoo \bigl((x_e(1))_{e \in E \setminus T},\; (x_e(0))_{e \in s^{-1}(\partial\Gamma) \setminus \{e_0\}},\; (x_e(1))_{e \in t^{-1}(\partial\Gamma)}\bigr) \nonumber
\end{align}
is an isomorphism.
\end{Lem}

\begin{proof}
Let $(a,b)\in C^0(I,W)\times W^{\Gamma\sint}$. 
We construct a solution to $R(x) = (a, b)$ as follows.
Choose a root $r_0 \in \partial\Gamma$ and a spanning tree $T \subset E$ rooted at $r_0$. 
For each edge $e \in E \setminus T$, choose any $x_e$ solving $\dot{x}_e + B_e x_e = a_e$. 
We then determine $x_e$ for $e \in T$ inductively, starting from the leaves of $T$ and moving toward the root. 
At a vertex $v$, all edges adjacent to $v$ except one are already known, i.e.\ it only remains to determine the solution on the parent edge $e_v \in T$ of $v$.
The Kirchhoff constraint (the second component of $R(x) = (a,b)$) uniquely determines the missing endpoint value of $x_{e_v}$, after which $\dot{x}_{e_v} + B_{e_v}x_{e_v} = a_{e_v}$ admits a unique $C^1$ solution on $I_{e_v}$. 
Proceeding inductively defines $x$ on all edges. 
The boundary vertex $r_0$ at the root imposes no additional condition, so the construction terminates. 

The initial values determining the solution $x$ correspond exactly to the map \eqref{di992oem}, which is therefore surjective. 
To compute $\dim \ker R$, note that the first component of $R(x) = 0$ implies that $x$ is determined by its initial conditions on each edge, giving $|E|\dim W$ degrees of freedom. 
The second component imposes $|\Gamma\sint|\dim W$ independent constraints, so $\dim \ker R = (|E| - |\Gamma\sint|)\dim W$. 
Since $|E| = |E \setminus T| + |V \setminus \{r_0\}| = |E \setminus T| + |s^{-1}(\partial\Gamma) \setminus \{e_0\}| + |t^{-1}(\partial\Gamma)| + |\Gamma\sint|$, the map \eqref{di992oem} is a bijection by dimension count.
\end{proof}

\begin{Prop}
The subset $\A(\Gamma) \subset \B(\Gamma)$ is a Banach submanifold.
\end{Prop}

\begin{proof}
By Lemma \ref{qanw51t3}, the differential
\begin{align*}
d\psi_A : \B(\Gamma) &\too C^0(I, \g) \times \g^{\Gamma\sint},\\
X &\mtoo 
\left(
\dot{X}_1 + [A_0, X_1] + [X_0, A_1],
\left(
\sum_{e \in s^{-1}(v)} X_1^e(0) - \sum_{e \in t^{-1}(v)} X_1^e(1)
\right)_{v \in \Gamma\sint}
\right)
\end{align*}
is surjective for all $A \in \psi^{-1}(0)$. 
Hence $0$ is a regular value of $\psi$, and $\A(\Gamma)$ is a Banach submanifold of $\B(\Gamma)$ by the inverse mapping theorem.
\end{proof}

\subsection{A free and proper action}

We now show that the action of $\G_0(\Gamma)$ on $\A(\Gamma)$ is free and proper.
This will be essential for the construction of local slices in the next subsection.

\begin{Prop}\label{xo75edv6}
The action of $\G_0(\Gamma)$ on $\A(\Gamma)$ is free.
\end{Prop}

\begin{proof}
Let $g \in \G_0(\Gamma)$ and $A \in \A(\Gamma)$ be such that $g \cdot A = A$. 
Since $\partial\Gamma \neq \emptyset$, there exists an edge $e \in E$ adjacent to $\partial\Gamma$. 
Assume $s(e) \in \partial\Gamma$; the other case is analogous. 
We then have $g_e \cdot A_e = A_e$ and $g_e(0) = 1$, which implies that $g_e$ satisfies the linear ordinary differential equation $\dot{g}_e = [A_0^e, g_e]$ with initial condition $g_e(0) = 1$. 
By uniqueness of solutions, $g_e \equiv 1$. 
Applying the same argument to any edge $f$ adjacent to $t(e)$ yields $g_f \equiv 1$, and by connectedness of $\Gamma$, we conclude that $g \equiv 1$.
\end{proof}

To establish properness, we first record the following useful fact.

\begin{Lem}\label{0b45jf95}
Let $\mathcal{P} \coloneqq \{g \in C^2([0, 1], G) : g(0) = 1\}$. 
Then the map
\[
\Phi : \mathcal{P} \too C^1([0, 1], \g), 
\qquad 
\Phi(g) = -\dot{g} g^{-1},
\]
is a homeomorphism of Banach manifolds.
\end{Lem}

\begin{proof}
Given $v\in C^1([0,1],\g)$, the equation $\Phi(g)=v$ is equivalent to the initial value problem
\[
\dot g + vg = 0,\qquad g(0)=1.
\]
By standard existence and uniqueness, this has a unique solution $g\in C^2([0,1],G)$, yielding a well-defined inverse $\Psi:C^1([0,1],\g)\to\mathcal{P}$.

It remains to prove that $\Psi$ is continuous. Let $v_n\to v$ in $C^1([0,1],\g)$ and set $g_n=\Psi(v_n)$, $g=\Psi(v)$.
Write $h_n\coloneqq g_n-g$. Then
\begin{equation}\label{l6xi4rz2}
\dot h_n=-v_n g_n+vg=-v_n(g+h_n)+vg=-v_n h_n-(v_n-v)g,
\end{equation}
so, for $t\in[0,1]$,
\[
h_n(t)=-\int_0^t\big(v_n(s)h_n(s)+(v_n(s)-v(s))g(s)\big)ds.
\]
Let $M\coloneqq \sup_n\|v_n\|_\infty$ and $C\coloneqq \|g\|_\infty$. Taking norms gives
\[
\|h_n(t)\|\le C\|v_n-v\|_\infty t+M\int_0^t\|h_n(s)\|ds.
\]
By Grönwall’s inequality,
\[
\|h_n(t)\|\le t e^{tM} \|v_n-v\|_\infty\quad\text{for all }t\in[0,1],
\]
hence $\|h_n\|_\infty\to 0$.
By \eqref{l6xi4rz2}, we have
\[
\|\dot h_n\|_\infty\le M\|h_n\|_\infty+C\|v_n-v\|_\infty\longrightarrow 0.
\]
Differentiating once more,
\[
\ddot h_n = -\dot v_n h_n - v_n \dot h_n - (\dot v_n-\dot v)g - (v_n-v)\dot g,
\]
whence, with $M_1\coloneqq \sup_n\|\dot v_n\|_\infty$ and $\|\dot g\|_\infty\le \|v\|_\infty\|g\|_\infty$,
\[
\|\ddot h_n\|_\infty\le M_1\|h_n\|_\infty + M\|\dot h_n\|_\infty
+ \|g\|_\infty\|\dot v_n-\dot v\|_\infty + \|\dot g\|_\infty\|v_n-v\|_\infty \longrightarrow 0.
\]
Thus $g_n\to g$ in $C^2([0,1],G)$, proving that $\Psi$ is continuous and therefore that $\Phi$ is a homeomorphism.
\end{proof}

\begin{Prop}
The action of $\G_0(\Gamma)$ on $\A(\Gamma)$ is proper.
\end{Prop}

\begin{proof}
Let $g_n \in \G_0(\Gamma)$ and $A_n \in \A(\Gamma)$ be such that both $(A_n)_{n=1}^\infty$ and $(g_n \cdot A_n)_{n=1}^\infty$ converge in $\A(\Gamma)$. 
We show, by induction on the edges of $\Gamma$, that $g_n^e$ converges in $C^2([0, 1], G)$ for all $e \in E$.

Without loss of generality, assume that $\partial\Gamma^- \neq \emptyset$, and let $e \in E$ be such that $s(e) \in \partial\Gamma$. 
By Lemma \ref{0b45jf95}, there exists $h_n^e \in \mathcal{P}$ satisfying
\[
A_n^e = (\Phi(h_n^e), *) = h_n^e \cdot (0, *).
\]
Since $A_n^e$ converges in $C^1([0, 1], \g)$ and $\Phi$ is a homeomorphism, it follows that $h_n^e$ converges in $\mathcal{P}$. 
Now observe that
\[
g_n^e \cdot A_n^e = g_n^e h_n^e \cdot (0, *) = (\Phi(g_n^e h_n^e), *),
\]
and since $g_n^e \cdot A_n^e$ converges, so does $\Phi(g_n^e h_n^e)$. 
By Lemma \ref{0b45jf95} again, this implies that $g_n^e h_n^e$ converges in $\mathcal{P}$, and hence $g_n^e$ converges in $C^2([0, 1], G)$.

Next, let $f$ be any edge adjacent to $t(e)$; assume $s(f) = t(e)$ (the other case is analogous). 
As before, we can write $A_n^f = h_n^f \cdot (0, *)$ with $h_n^f$ converging in $\mathcal{P}$. 
Set $a_n \coloneqq g_n^e(1) = g_n^f(0) \in G$. Since $g_n^e$ converges in $\mathcal{P}$, the sequence $(a_n)_{n=1}^\infty$ converges in $G$. 
Define $k_n^f \coloneqq a_n^{-1} g_n^f h_n^f \in \mathcal{P}$. We have
\[
g_n^f \cdot A_n^f 
= g_n^f h_n^f \cdot (0, *) 
= a_n k_n^f \cdot (0, *) 
= (\Ad_{a_n}\Phi(k_n^f), *),
\]
and since both $a_n$ and $g_n^f \cdot A_n^f$ converge, it follows that $\Phi(k_n^f)$ converges, hence $k_n^f$ converges in $\mathcal{P}$. 
Consequently, $g_n^f$ converges in $C^2([0, 1], G)$.

Repeating this argument inductively along adjacent edges and using the connectedness of $\Gamma$, we conclude that $g_n^e$ converges in $C^2([0, 1], G)$ for all $e \in E$. 
Therefore, $g_n$ converges in $\G_0(\Gamma)$, and the action is proper.
\end{proof}

\subsection{Existence of slices}\label{75drzcuc}

We now prove the existence of local slices for the action of $\G_0(\Gamma)$ on $\A(\Gamma)$, culminating in the proof of Theorem \ref{pa1025dm}. 
Our argument is inspired by the approach of Kronheimer \cite{kronheimer1988hyperkahler}.

For $A \in \A(\Gamma)$, the tangent space is
\[
T_A\A(\Gamma)
= 
\{ Y \in \K(\Gamma) : \dot{Y}_1 + [Y_0, A_1] + [A_0, Y_1] = 0 \}.
\]
Let
\[
C^2_0(I, \g) 
\coloneqq \operatorname{Lie}(\G_0(\Gamma)).
\]
That is, $C^2_0(I, \g)$ consists of all elements $u \in C^2(I, \g)$ for which there exist values $u(v) \in \g$ for each $v \in V$ such that 
$u_e(0) = u(s(e))$ and $u_e(1) = u(t(e))$ for all $e \in E$, 
and $u(v) = 0$ for all $v \in \partial\Gamma$.
The infinitesimal action of $\G_0(\Gamma)$ at $A$ is the linear map
\[
D_A : C^2_0(I, \g) \too T_A\A(\Gamma),
\qquad
u \mtoo ([u, A_0] - \dot{u}, [u, A_1]).
\]
There is a jointly continuous, non-degenerate pairing between $C^2_0(I, \g)$ and $C^0(I, \g) \times \g^{\Gamma\sint}$ given by
\begin{equation}\label{v5stgm73}
C^2_0(I, \g) \times (C^0(I, \g) \times \g^{\Gamma\sint}) \too \R, \quad \bigl\langle u, (Y, Z) \bigr\rangle
\coloneqq
\int_I \langle u(t), Y(t) \rangle dt
+ \sum_{v \in \Gamma\sint}
  \langle u(v), Z(v) \rangle.
\end{equation}
The tangent space $T_A\A(\Gamma)$ also carries the natural inner product
\[
\langle Y, Z \rangle
=
\int_I
  \bigl(
    \langle Y_0(t), Z_0(t) \rangle
    + \langle Y_1(t), Z_1(t) \rangle
  \bigr)dt.
\]
We will show in the next proposition that the operator
\begin{align*}
D_A^* : T_A\A(\Gamma) &\too C^0(I, \g) \times \g^{\Gamma\sint}\\
Y &\mtoo
\left(
  \dot{Y}_0 + [A_0, Y_0] + [A_1, Y_1],
  \left(
    \sum_{e \in s^{-1}(v)} Y_0^e(0)
    - \sum_{e \in t^{-1}(v)} Y_0^e(1)
  \right)_{v \in \Gamma\sint}
\right)
\end{align*}
is an adjoint of $D_A$ with respect to these pairings.

\begin{Prop}
We have
\[
\langle D_Au, Y \rangle = \langle u, D_A^*Y \rangle,
\]
for all $u \in C^2_0(I, \g)$ and $Y \in T_A\A(\Gamma)$.
\end{Prop}

\begin{proof}
This is a straightforward computation using integration by parts and the $\Ad$-invariance of the inner product on $\g$. 
Indeed,
\begin{align*}
\langle D_Au, Y \rangle
&=
\int_I
  \bigl(
    \langle [u, A_0] - \dot{u}, Y_0 \rangle
    + \langle [u, A_1], Y_1 \rangle
  \bigr) \\
&=
\int_I
  \langle u, \dot{Y}_0 + [A_0, Y_0] + [A_1, Y_1] \rangle
  + \sum_{e \in E}
    \bigl(
      \langle u_e(0), Y_0^e(0) \rangle
      - \langle u_e(1), Y_0^e(1) \rangle
    \bigr) \\
&=
\int_I
  \langle u, \dot{Y}_0 + [A_0, Y_0] + [A_1, Y_1] \rangle
  + \sum_{v \in \Gamma\sint}
    \Bigl\langle
      u(v),
      \sum_{e \in s^{-1}(v)} Y_0^e(0)
      - \sum_{e \in t^{-1}(v)} Y_0^e(1)
    \Bigr\rangle \\
&= \langle u, D_A^*Y \rangle.\qedhere
\end{align*}
\end{proof}

A natural candidate for a slice at $A$ is a submanifold $S$ whose tangent space at $A$ is orthogonal to the $\G_0(\Gamma)$-orbit through $A$. 
Since $\im D_A$ is the infinitesimal orbit, this orthogonal complement is $\ker D_A^*$.
To realize this decomposition, we will need the following analytic result. 
For $A_0 \in C^1(I, \g)$ and $u \in C^2(I, \g)$, set $\nabla_{A_0} u \coloneqq \dot{u} + [A_0, u]$.

\begin{Prop}\label{derirtj5}
The operator
\begin{align}
L_A \coloneqq D_A^*D_A : C^2_0(I, \g) &\too C^0(I, \g) \times \g^{\Gamma\sint}
\label{jnpsk711} \\
u &\mtoo
\left(
  -\nabla_{A_0}^2 u - \ad_{A_1}^2 u,\;
  \left(
    \sum_{e \in s^{-1}(v)} \dot{u}^e(0)
    - \sum_{e \in t^{-1}(v)} \dot{u}^e(1)
  \right)_{v \in \Gamma\sint}
\right)
\nonumber
\end{align}
is an isomorphism of Banach spaces. 
In particular,
\[
T_A\A(\Gamma)
= \im D_A \oplus \ker D_A^*
\]
is a closed, topological direct sum.
\end{Prop}

We establish this through a sequence of lemmas.

\begin{Lem}\label{ia3kcr17}
The operator $D_A^*$ is surjective, and
\[
\dim \ker D_A^* = 2(|E| - |\Gamma\sint|)\dim \g.
\]
\end{Lem}

\begin{proof}
This follows from Lemma \ref{qanw51t3} with $W = \g \times \g$.
\end{proof}

\begin{Lem}\label{fkr5s377}
The operator \eqref{jnpsk711} is injective.
\end{Lem}

\begin{proof}
If $u \in C^2_0(I, \g)$ satisfies $D_A^*D_Au = 0$, then $\langle D_A^*D_Au, u \rangle = \|D_Au\|^2 = 0$, so $D_Au = 0$. 
In particular, $u$ satisfies $\dot{u}_e = [u_e, A_0^e]$ on each edge $e \in E$. 
Because $\partial\Gamma \neq \emptyset$ and $u(v) = 0$ for all $v \in \partial\Gamma$, the connectedness of $\Gamma$ implies $u \equiv 0$, exactly as in the proof of Proposition \ref{xo75edv6}.
\end{proof}

\begin{Lem}\label{w8g2jo3c}
Setting $A=0$, the operator
\begin{align*}
L_0 : C^2_0(I,\g) &\too C^0(I,\g)\times \g^{\Gamma\sint}\\
u &\mtoo 
\left(
  -\ddot{u},\;
  \left(\sum_{e\in s^{-1}(v)}\dot{u}^e(0)-\sum_{e\in t^{-1}(v)}\dot{u}^e(1)\right)_{v\in\Gamma\sint}
\right)
\end{align*}
is an isomorphism of Banach spaces.
\end{Lem}

\begin{proof}
The map is bounded, so it suffices to prove bijectivity. 
Injectivity has already been shown in Lemma \ref{fkr5s377}.
For surjectivity, let $(v,y)\in C^0(I,\g)\times \g^{\Gamma\sint}$. 
For each edge $e\in E$, pick $w_e$ solving $-\ddot{w}_e=v_e$ with $w_e(0)=w_e(1)=0$.
Any solution to $L_0(u)=v$ then has the form
\[
u_e(t)=w_e(t)+a_e+b_e t,\qquad a_e,b_e\in\g.
\]
It remains to choose $(a_e,b_e)_{e\in E}$ so that
\begin{enumerate}
\item the maps $(u_e)_{e \in E}$ define an element of $C^2_0(I, \g)$, i.e.\ there exist values $(c_v)_{v \in \Gamma\sint} \in \g^{\Gamma\sint}$ such that
\begin{itemize}
    \item $u_e(0) = 0$ if $s(e) \in \partial\Gamma$,
    \item $u_e(1) = 0$ if $t(e) \in \partial\Gamma$,
    \item $u_e(0) = c_{s(e)}$ if $s(e) \in \Gamma\sint$,
    \item $u_e(1) = c_{t(e)}$ if $t(e) \in \Gamma\sint$,
\end{itemize}
and
\item $\sum_{e\in s^{-1}(v)}\dot{u}_e(0)-\sum_{e\in t^{-1}(v)}\dot{u}_e(1)=y_v$ for all $v\in\Gamma\sint$.
\end{enumerate}
These are linear algebraic conditions which reduce to
\begin{align}
a_e&=0 &&\text{for all } e\in s^{-1}(\partial\Gamma), \label{c91q6wa7}\\
a_e+b_e&=0 &&\text{for all } e\in t^{-1}(\partial\Gamma), \\
a_e &= c_{s(e)} &&\text{for all } e \in s^{-1}(\Gamma\sint), \label{y08iimjv}\\
a_e+b_e&= c_{t(e)} &&\text{for all } e \in t^{-1}(\Gamma\sint), \label{792h1y84}\\
\sum_{e\in s^{-1}(v)} b_e-\sum_{e\in t^{-1}(v)} b_e&=z_v 
&&\text{for } v\in\Gamma\sint, \label{wgpmyttd}
\end{align}
where $z_v \coloneqq y_v - \sum_{e\in s^{-1}(v)}\dot{w}_e(0) + \sum_{e\in t^{-1}(v)}\dot{w}_e(1)\in\g$ is fixed. From \eqref{c91q6wa7}--\eqref{792h1y84}, $a_e$ and $b_e$ are completely determined by $c \in \g^{\Gamma\sint}$. Then \eqref{wgpmyttd} becomes an equation of the form
\begin{equation}\label{spqm09u4}
    M(c) = z,
\end{equation}
for some linear operator $M : \g^{\Gamma\sint} \to \g^{\Gamma\sint}$.
It then suffices to show that \eqref{spqm09u4} admits a solution $c \in \g^{\Gamma\sint}$ for any $z \in \g^{\Gamma\sint}$.
We claim that $M$ is injective.
Indeed, if $M(c)=0$, take $w_e\equiv 0$ and form $u_e(t)=a_e+b_e t$ from $c$ as above. 
Then $L_0(u)=0$, and by injectivity of $L_0$ we get $u\equiv 0$, hence $c=0$. 
Therefore $M$ is injective, hence bijective, so there is a unique $c$ solving \eqref{spqm09u4}. 
This produces $(a_e,b_e)$ and thus $u$ with $L_0(u)=v$.
\end{proof}

\begin{proof}[Proof of Proposition \ref{derirtj5}]
Write $L_A=L_0+K$, where $L_0$ is as in Lemma \ref{w8g2jo3c} and
\[
K(u)=\bigl(-2\ad_{A_0}\dot{u}-\ad_{\dot{A}_0}u-\ad_{A_0}^2u-\ad_{A_1}^2u,\;0\bigr).
\]
The map $K$ is compact, since it factors as
\[
\begin{tikzcd}
C^2_0(I,\g) \arrow[hook]{r} & C^1(I,\g) \arrow{r} & C^0(I,\g)\times \g^{\Gamma\sint},
\end{tikzcd}
\]
where the first map is compact by Arzelà--Ascoli (see, e.g., \cite[Thm.\ 1.34]{AdamsFournier}) and the second map is bounded.

Set $T\coloneqq L_0^{-1}K$, which is a compact operator on $C^2_0(I,\g)$. 
Then $L_A=L_0(I+T)$. 
By the Fredholm alternative, either $I+T$ is invertible or there exists $u \ne 0$ with $(I+T)u=0$. 
In the latter case, $L_Au=0$, which contradicts Lemma \ref{fkr5s377} (injectivity of $L_A$). 
Hence $I+T$ is invertible, and therefore $L_A$ is an isomorphism.
\end{proof}

With Proposition \ref{derirtj5} in hand, we construct a local slice for the $\G_0(\Gamma)$-action at a fixed $A\in\A(\Gamma)$. Consider the gauge-fixing map
\[
F : \G_0(\Gamma)\times \A(\Gamma)\longrightarrow C^0(I,\g)\times \g^{\Gamma\sint},
\qquad
F(g,B)=D_A^*(g\cdot B-A).
\]
We have $F(1,A)=0$ and, for $u\in T_1\G_0(\Gamma)=C^2_0(I,\g)$,
\[
dF_{(1,A)}(u,0)=D_A^*D_Au.
\]
By Proposition \ref{derirtj5}, $D_A^*D_A$ is an isomorphism. Hence, by the implicit function theorem, there exist neighbourhoods $U\subset\A(\Gamma)$ of $A$ and $V\subset\G_0(\Gamma)$ of $1$ and a unique smooth map $s: U\to V$ such that $F(s(B),B)=0$ for all $B\in U$.
Define the candidate slice
\[
S\coloneqq s^{-1}(1)
=\{B\in U:F(1,B)=0\}
=\{A+a\in U: D_A^*a=0\}.
\]

\begin{Lem}
After shrinking $U$ if necessary, $S$ is a Banach submanifold of $\A(\Gamma)$.
\end{Lem}

\begin{proof}
It suffices to show that $ds_A: T_A\A(\Gamma)\to T_1\G_0(\Gamma)=C^2_0(I,\g)$ is surjective. For $v\in T_A\A(\Gamma)$, we have
\[
dF_{(1,A)}(ds_A(v),v)=0,
\]
hence
\begin{equation}\label{xgsveh0y}
D_A^*D_A ds_A(v)=-dF_{(1,A)}(0,v)=-D_A^*v.
\end{equation}
Since $D_A^*$ is surjective (Lemma \ref{ia3kcr17}), given any $u\in C^2_0(I,\g)$ we may choose $v$ with $D_A^*v=-D_A^*D_Au$. Applying \eqref{xgsveh0y} and inverting $D_A^*D_A$ (Proposition \ref{derirtj5}) yields $ds_A(v)=u$. Thus $ds_A$ is surjective, and $S$ is a Banach submanifold.
\end{proof}

Equation \eqref{xgsveh0y} also shows $\ker ds_A=\ker D_A^*$, hence
\begin{equation}\label{63cdb20b}
T_AS=\ker D_A^*.
\end{equation}
Consider the action map
\[
\Psi : \G_0(\Gamma)\times S \longrightarrow \A(\Gamma),\qquad \Psi(g,B)=g\cdot B.
\]
Its differential at $(1,A)$ is
\[
d\Psi_{(1,A)}: T_1\G_0(\Gamma)\times T_AS \longrightarrow T_A\A(\Gamma),
\qquad
(u,\xi)\longmapsto D_Au+\xi.
\]
By Proposition \ref{derirtj5} and \eqref{63cdb20b}, this is the (topological) direct sum decomposition $T_A\A(\Gamma)=\im D_A\oplus \ker D_A^*$.
Therefore, after shrinking $U$ and $V$ if necessary, $\Psi$ restricts to a diffeomorphism
\[
\Psi: V\times S \stackrel{\cong}{\longrightarrow} \A(\Gamma)
\]
onto a neighbourhood of $A$.

\begin{Lem}\label{e30z2nqr}
After possibly shrinking $U$, if $g\cdot S\cap S\neq\emptyset$ then $g=1$.
\end{Lem}

\begin{proof}
By properness of the $\G_0(\Gamma)$-action, we may shrink $U$ so that $g\cdot U\cap U\neq\emptyset$ implies $g\in V$. If $g\cdot S\cap S\neq\emptyset$, pick $B\in S$ with $g\cdot B\in S$. Since $\Psi$ is injective on $V\times S$, we must have $g=1$.
\end{proof}

By Lemma \ref{e30z2nqr}, $\Psi$ is a diffeomorphism from $\G_0(\Gamma)\times S$ onto the open neighbourhood $W=\G_0(\Gamma)\cdot S$ of $A$. Thus $S$ is a slice for the $\G_0(\Gamma)$-action at $A$. This completes the construction of the smooth structure asserted in Theorem \ref{pa1025dm}. The dimension formula follows from \eqref{63cdb20b} together with Lemma \ref{ia3kcr17}.

\section{The symplectic structure}\label{vpbytyos}

We now construct a canonical symplectic structure on $\M(\Gamma)$ via symplectic reduction in the infinite-dimensional setting of Marsden--Weinstein \cite{MarsdenWeinstein}, and verify the hypotheses needed in our context. While more general treatments of infinite-dimensional symplectic reduction exist---see, e.g., \cite{DiezRudolph2024SymplecticReduction}---the original work \cite{MarsdenWeinstein} suffices for our purpose.

\begin{Thm}[Symplectic structure]\label{0hdsx99j}
Let $\Gamma$ be a connected quiver with non-empty boundary. Then the map $\psi$ from \eqref{bbqchqb6} is a moment map for the action of $\G_0(\Gamma)$ on $\B(\Gamma)$ and the Lax--Kirchhoff moduli space
\[
\M(\Gamma)=\psi^{-1}(0)/\G_0(\Gamma)
\]
is a finite-dimensional smooth symplectic manifold.
\end{Thm}

On the Banach space $\B(\Gamma)$ we use the constant (weakly) non-degenerate $2$-form
\[
\omega(X,Y)\coloneqq \int_I \bigl(\langle X_0(t),Y_1(t)\rangle
-\langle X_1(t),Y_0(t)\rangle\bigr)dt,
\qquad X,Y\in \B(\Gamma).
\]
By $\Ad$-invariance of $\langle\cdot,\cdot\rangle$ on $\g$, the action of $\G_0(\Gamma)$ on $\B(\Gamma)$ preserves $\omega$. Using the pairing \eqref{v5stgm73}, a moment map for this action may be taken with values in $C^0(I,\g)\times \g^{\Gamma\sint}$.

\begin{Prop}
The map $\psi$ in \eqref{bbqchqb6} is a moment map for the action of $\G_0(\Gamma)$ on $\B(\Gamma)$.
\end{Prop}

\begin{proof}
Let $u\in \Lie\bigl(\G_0(\Gamma)\bigr)$, $A\in \B(\Gamma)$, and $X\in T_A\B(\Gamma)=\B(\Gamma)$. Writing
\[
\langle u,\psi\rangle(A)
=\int_I \langle u,\dot A_1+[A_0,A_1]\rangle
+\sum_{v\in \Gamma\sint}
  \Bigl\langle u(v),
  \sum_{e\in s^{-1}(v)}A_1^e(0)-\sum_{e\in t^{-1}(v)}A_1^e(1)\Bigr\rangle,
\]
we get
\[
d\langle u,\psi\rangle_A(X)
=\int_I \langle u,\dot X_1+[A_0,X_1]+[X_0,A_1]\rangle
+\sum_{v\in \Gamma\sint}
  \Bigl\langle u(v),
  \sum_{e\in s^{-1}(v)}X_1^e(0)-\sum_{e\in t^{-1}(v)}X_1^e(1)\Bigr\rangle.
\]
Since $u(v)=0$ for $v\in\partial\Gamma$, the vertex sum reduces to
\[
\sum_{e\in E}\bigl(\langle u_e(0),X_1^e(0)\rangle-\langle u_e(1),X_1^e(1)\rangle\bigr),
\]
which is precisely the negative of the boundary term obtained by integrating $\int_I \langle u,\dot X_1\rangle$ by parts. Hence
\[
d\langle u,\psi\rangle_A(X)
=\int_I \bigl(
  \langle [u,A_0]-\dot u,X_1\rangle
 -\langle [u,A_1],X_0\rangle
\bigr)
=\omega\bigl(D_A(u),X\bigr),
\]
which is the defining identity for a moment map.
\end{proof}

At this point, we have a free and proper Hamiltonian action of $\G_0(\Gamma)$ on
$\B(\Gamma)$ with moment map $\psi$. In finite dimensions, these hypotheses already imply that the reduced space $\psi^{-1}(0)/\G_0(\Gamma)$ is symplectic.
In the infinite-dimensional setting of \cite{MarsdenWeinstein}, two additional conditions must be verified:
\begin{enumerate}[label={(\roman*)}]
\item the action of $\G_0(\Gamma)$ on $\psi^{-1}(0)$ admits local slices, and
\item for each $A\in \psi^{-1}(0)$, the tangent space $T_A(\G_0(\Gamma) \cdot A)$ is the $\omega$-orthogonal complement of $T_A\psi^{-1}(0)$.
\end{enumerate}
As explained in \cite[p.\ 123]{MarsdenWeinstein}, the definition of a moment map ensures that $T_A\psi^{-1}(0)$ is always the $\omega$-orthogonal to $T_A(\G_0(\Gamma)\cdot A)$, but in infinite dimensions the converse (ii) is not automatic and must be proved separately. Condition (i) was established in Section \ref{75drzcuc}; the next lemma proves (ii).

\begin{Lem}
Let $A\in \A(\Gamma)$ and $Z\in \B(\Gamma)$ satisfy $\omega(Z,Y)=0$ for all $Y\in T_A\A(\Gamma)$. Then there exists $u\in \Lie(\G_0(\Gamma))$ such that $Z=D_A(u)$.
\end{Lem}

\begin{proof}
Assume $\partial\Gamma^-\neq\emptyset$; the opposite case is analogous.  
Pick $r_0\in\partial\Gamma^-$ and let $T\subset E$ be a spanning tree rooted at $r_0$.  
Arguing as in the proof of Lemma \ref{qanw51t3}, but propagating from the root to the leaves, we construct $u_e\in C^2([0,1],\g)$ for each $e\in T$, and $c_v \in \g$ for each $v \in V$, satisfying
\begin{enumerate}[label={(\alph*)}]
\item $Z_0^e=[u_e,A_0^e]-\dot u_e$ for all $e\in T$, and
\item $u_e(0) = c_{s(e)}$ and $u_e(1) = c_{t(e)}$ for all $e \in T$,
\end{enumerate}
with $c_{r_0} = 0$.
For every edge $e\notin T$, choose $u_e$ with $u_e(0)=c_{s(e)}$, producing $u=(u_e)_{e\in E}\in C^2(I,\g)$ satisfying $Z_0=[u,A_0]-\dot u$.
Because $t^{-1}(\partial\Gamma)\subset T$, we already have
$u_e(1)=c_{t(e)}$ for $e\in t^{-1}(\partial\Gamma)$.
To complete the argument we must show:
\begin{enumerate}
\item $u_e(1)=c_{t(e)}$ for all $e\notin T$;
\item $c_v=0$ for all $v\in\partial\Gamma$;
\item $Z_1=[u,A_1]$.
\end{enumerate}
For any $Y\in T_A\A(\Gamma)$ we compute
\begin{align*}
\omega(Z,Y)
&=\int_I \bigl(\langle [u,A_0]-\dot u, Y_1\rangle-\langle Z_1,Y_0\rangle\bigr)\\
&=\int_I \bigl(\langle u,[A_0,Y_1]+\dot Y_1\rangle-\langle Z_1,Y_0\rangle\bigr)
  +\sum_{e\in E}\bigl(\langle u_e(0),Y_1^e(0)\rangle-\langle u_e(1),Y_1^e(1)\rangle\bigr)\\
&=\int_I \langle [u,A_1]-Z_1,Y_0\rangle
  +\sum_{v\in V} \left(
    \sum_{e\in s^{-1}(v)} \langle u_e(0),Y_1^e(0)\rangle
   - \sum_{e\in t^{-1}(v)} \langle u_e(1),Y_1^e(1)\rangle 
  \right).
\end{align*}
The vertex sum then decomposes as
\begin{align*}
&\sum_{v \in \partial \Gamma^-} \sum_{e \in s^{-1}(v)} \langle u_e(0), Y^e_1(0) \rangle - \sum_{v \in \partial \Gamma^+} \sum_{e \in t^{-1}(v)} \langle u_e(1), Y^e_1(1) \rangle \\
&\qquad + \sum_{v \in \Gamma\sint}\left(\sum_{e \in s^{-1}(v)} \langle u_e(0), Y^e_1(0) \rangle - \sum_{e \in t^{-1}(v)} \langle u_e(1), Y^e_1(1) \rangle\right) \\
&= \sum_{e \in s^{-1}(\partial\Gamma)} \langle c_{s(e)}, Y^e_1(0) \rangle - \sum_{e \in t^{-1}(\partial\Gamma)} \langle c_{t(e)}, Y^e_1(1) \rangle \\
&\qquad + \sum_{v \in \Gamma\sint} \left(\left\langle c_v , \sum_{e \in s^{-1}(v)} Y^e_1(0) - \sum_{e \in t^{-1}(v) \cap T} Y^e_1(1) \right\rangle - \sum_{e \in t^{-1}(v) \setminus T} \langle u_e(1), Y^e_1(1) \rangle\right) \\
&= \sum_{e \in s^{-1}(\partial\Gamma)} \langle c_{s(e)}, Y^e_1(0) \rangle - \sum_{e \in t^{-1}(\partial\Gamma)} \langle c_{t(e)}, Y^e_1(1) \rangle \\
&\qquad + \sum_{v \in \Gamma\sint} \left(\left\langle c_v , \sum_{e \in t^{-1}(v) \setminus T} Y^e_1(1) \right\rangle - \sum_{e \in t^{-1}(v) \setminus T} \langle u_e(1), Y^e_1(1) \rangle\right) \\
&= \sum_{e \in s^{-1}(\partial\Gamma)} \langle c_{s(e)}, Y^e_1(0) \rangle - \sum_{e \in t^{-1}(\partial\Gamma)} \langle c_{t(e)}, Y^e_1(1) \rangle + \sum_{e \in E \setminus T} \langle c_{t(e)} - u_e(1), Y^e_1(1) \rangle
\end{align*}
By Lemma \ref{qanw51t3}, we can choose $Y\in T_A\A(\Gamma)$ with
$Y_0=[u,A_1]-Z_1$ and
\[
Y^e_1(0) = c_{s(e)} \text{ if } e \in s^{-1}(\partial\Gamma) \setminus \{e_0\},
\qquad\text{and}\qquad
Y^e_1(1) =
\begin{cases}
c_{t(e)} - u_e(1) &\text{; if } e \in E \setminus T \\
-c_{t(e)} &\text{; if } e \in t^{-1}(\partial \Gamma).
\end{cases}
\]
Substituting this $Y$ into the above formula yields
\begin{align*}
\omega(Z,Y)
&=\int_I \|[u,A_1]-Z_1\|^2
 +\sum_{v\in\partial\Gamma\setminus\{r_0\}} \|c_v\|^2
 +\sum_{e\in E\setminus T} \|c_{t(e)}-u_e(1)\|^2.
\end{align*}
Since $\omega(Z,Y)=0$ by assumption, all terms vanish, forcing $[u,A_1]=Z_1$, $u_e(1)=c_{t(e)}$, and $c_v=0$ for all $v \in \partial\Gamma$.
Thus $u\in \Lie(\G_0(\Gamma))$ and $Z=D_A(u)$.
\end{proof}

By \cite{MarsdenWeinstein}, the quotient $\M(\Gamma)=\psi^{-1}(0)/\G_0(\Gamma)$ inherits a smooth symplectic structure, completing the proof of Theorem \ref{0hdsx99j}.

\section{The isomorphism with $T^*G^E\sll{} G^{\Gamma\sint}$}\label{rj2bih75}

Let $\Gamma=(V,E,s,t)$ be a connected quiver with non-empty boundary. The goal of this section is to identify $\M(\Gamma)$ with a symplectic reduction of $T^*G^E$ by an action of $G^{\Gamma\sint}$. This will also show that the action of $G^{\partial\Gamma}$ on $\M(\Gamma)$ is Hamiltonian.

Recall that the action of $G\times G$ on $G$ by left and right multiplications lifts to a Hamiltonian action on $T^*G$. Identifying $T^*G\cong G\times\g$ via left translations and the invariant inner product on $\g$, the action is
\[
(a,b)\cdot(g,x)=(a g b^{-1},\Ad_b x),
\]
with moment map
\[
\mu: T^*G\too \g\times\g,\qquad \mu(g,x)=(\Ad_g x,-x),
\]
and symplectic form
\[
\omega_{(g,x)}\bigl((u_1,v_1),(u_2,v_2)\bigr)=\langle u_1,v_2\rangle-\langle u_2,v_1\rangle+\langle x,[u_1,u_2]\rangle
\]
for $(g,x)\in T^*G$ and $(u_i,v_i)\in T_{(g,x)}T^*G\cong\g\times\g$; see, e.g., \cite[\S4.4]{AbrahamMarsden1978Foundations}. If $\seg$ denotes the quiver with a single edge, it is standard that $\M(\seg)\cong T^*G$ as Hamiltonian $G\times G$-spaces \cite{Hall2002Geometric,kronheimer1988hyperkahler,DancerSwann1996Hyperkahler,Bielawski2007Lie}, with explicit identification
\[
\M(\seg)\too T^*G,\qquad A\mtoo \bigl(g_A(1),A_1(0)\bigr),
\]
where $g_A\in C^2([0,1],G)$ is the unique solution to $\dot g_A+A_0 g_A=0$ with $g_A(0)=1$.

We now extend this to a general quiver. Consider the action of $G^V$ on $T^*G^E$ induced by the embedding
\[
G^V \longrightarrow (G\times G)^E,\qquad (g_v)_{v\in V}\longmapsto\bigl(g_{t(e)},g_{s(e)}\bigr)_{e\in E},
\]
that is,
\[
b\cdot(a,x)=\bigl(b_{t(e)} a_e b_{s(e)}^{-1},\Ad_{b_{s(e)}} x_e\bigr)_{e\in E}\qquad\text{for }b\in G^V,\ (a,x)\in T^*G^E.
\]
This action is Hamiltonian, and since $G^V\cong G^{\Gamma\sint}\times G^{\partial\Gamma}$, the commuting actions of the factors are Hamiltonian as well. In particular, the $G^{\Gamma\sint}$-action has moment map
\[
\nu: T^*G^E\too \g^{\Gamma\sint},\qquad
\nu(g,x)=\left(\sum_{e\in t^{-1}(v)}\Ad_{g_e} x_e-\sum_{e\in s^{-1}(v)} x_e\right)_{v\in\Gamma\sint}.
\]
If $\Gamma$ is connected with non-empty boundary, an induction argument as in the proof of Proposition \ref{xo75edv6} shows that the $G^{\Gamma\sint}$-action on $\nu^{-1}(0)$ is free. Hence $T^*G^E\sll{} G^{\Gamma\sint}$ is a Hamiltonian $G^{\partial\Gamma}$-space.

\begin{Thm}[Identification with $T^*G^E\sll{}G^{\Gamma\sint}$]\label{tis5zvw5}
Let $\Gamma=(V,E,s,t)$ be a connected quiver with non-empty boundary. The map
\begin{equation}\label{d3ofg4m5}
\varphi: \M(\Gamma)\too T^*G^E\sll{} G^{\Gamma\sint},\qquad
A\mtoo\bigl(g_A(1),A_1(0)\bigr),
\end{equation}
is a symplectomorphism, where $g_A\in C^2(I,G)$ is the unique solution to
$\dot g_A+A_0 g_A=0$ with $(g_A)_e(0)=1$ for all $e\in E$.
\end{Thm}

\begin{proof}
We first show that $\varphi$ is a diffeomorphism by constructing a smooth inverse $\psi$. Note that $g_A$ is uniquely characterized by the condition that $A = g_A \cdot (0, x)$ for some $x \in \g^E$ and $(g_A)_e(0) = 1$ for all $e \in E$. Indeed, if $A' = g_A^{-1} \cdot A$, then $g_A$ is defined precisely so that $A_0' = 0$. But the Lax equation is gauge invariant, so $\dot{A}_1' + [A_0', A_1'] = 0$ which implies that $A_1'$ is constant. Hence we may define an inverse by choosing, for each $(a,x)\in T^*G^E\sll{}G^{\Gamma\sint}$, a path $\gamma_a\in C^2(I,G)$ such that $(\gamma_a)_e(0)=1$ and $(\gamma_a)_e(1)=a^e$, and set
\[
\psi: T^*G^E\sll{}G^{\Gamma\sint}\too\M(\Gamma),\qquad
(a,x)\mtoo \gamma_a\cdot(0,x).
\]
The moment map condition $\nu(a, x) = 0$ ensures that $\psi(a, x)$ lies in $\A(\Gamma)$. This map is independent of the choice of $\gamma_a$, since any other choice $\tilde{\gamma}_a$ is of the form $\tilde{\gamma}_a = g \gamma_a$, where $g_e(0) = g_e(1) = 1$ for all $e \in E$ so that $g \in \G_0(\Gamma)$. It is also well-defined on the quotient by $G^{\Gamma\sint}$ since for $b \in G^{\Gamma\sint}$ and $b \cdot (a, x) = (c, y)$, we have $\gamma_c \cdot (0, y) = h \cdot \gamma_a \cdot (0, x)$, where $h = ((\gamma_c)_e b_{s(e)} (\gamma_a)_e^{-1})_{e \in E} \in \G_0(\Gamma)$. Hence $\psi$ is an inverse of $\varphi$. Smoothness of $\psi$ follows because
\[
\{\gamma \in C^2([0, 1], \g) : \gamma(0) = 1\} \too G, \quad \gamma \mtoo \gamma(1)
\]
is a surjective submersion of Banach manifolds and hence has smooth local sections. Thus, $\varphi$ is a diffeomorphism. 

We now show that $\psi$ is a symplectomorphism. 
We first note that
\[
d\psi_{(a, x)}(u, v) = (\Ad_{\gamma_a}u, \Ad_{\gamma_a}(t[u, x] + v)).
\]
Indeed, write $a = e^y$ and choose $\gamma_a(t) = e^{ty}$. Then $\gamma_{ae^{su}}(t) = e^{ty} e^{tsu}$ and a straightforward computation shows that
\[
d\psi_{(a, x)}(u, v) = \frac{d}{ds}\Big|_{s = 0} \gamma_{ae^{su}} \cdot (0, x + sv) = (\Ad_{\gamma_a}u, \Ad_{\gamma_a}(t [u, x] + v)).
\]
Thus,
\begin{align*}
(\psi^*\omega)_{(a, x)}((u_1, v_1), (u_2, v_2))
&= \int_I(\langle u_1, t[u_2, x] + v_2 \rangle - \langle u_2, t[u_1, x] + v_1 \rangle) \\
&= \langle u_1, v_2 \rangle - \langle u_2, v_1 \rangle + \langle x, [u_1, u_2] \rangle,
\end{align*}
as desired.
\end{proof}

The reduced space $T^*G^E\sll{}G^{\Gamma\sint}$ is a Hamiltonian $G^{\partial\Gamma}$-space, with moment map
\[
\lambda: T^*G^E\sll{}G^{\Gamma\sint}\too \g^{\partial\Gamma},\qquad
\lambda(a,x)_v=
\begin{cases}
\Ad_{a_e}x_e,& t(e)=v,\\
-x_e,& s(e)=v.
\end{cases}
\]
For a vertex $v\in\partial\Gamma^-$ and $A\in\B(\Gamma)$, set $A_i(v)=A_i^e(0)$ where $e$ is the unique edge with $s(e)=v$. Similarly, for $v\in\partial\Gamma^+$, let $A_i(v)=A_i^e(1)$ where $t(e)=v$. Define $\operatorname{sgn}(v)=\pm1$ for $v\in\partial\Gamma^{\pm}$. Under the diffeomorphism $\M(\Gamma)\cong T^*G^E\sll{}G^{\Gamma\sint}$, the $G^{\partial\Gamma}$-action corresponds to the action of $\G(\Gamma)/\G_0(\Gamma)\cong G^{\partial\Gamma}$ and the moment map corresponds to
\begin{equation}\label{e7o5w1qn}
\M(\Gamma)\too\g^{\partial\Gamma},\qquad
A\mtoo(\operatorname{sgn}(v)A_1(v))_{v\in\partial\Gamma}.
\end{equation}
This yields the following conclusion.

\begin{Thm}[Hamiltonian structure]
Let $\Gamma$ be a connected quiver with non-empty boundary. The action of $G^{\partial\Gamma}$ on $\M(\Gamma)$ is Hamiltonian with moment map \eqref{e7o5w1qn}. Moreover, the identification $\M(\Gamma) \cong T^*G^E \sll{} G^{\Gamma\sint}$ is an isomorphism of Hamiltonian spaces.\qed
\end{Thm}

\section{Quiver homotopies}\label{x63cjnor}

By thickening the topological realization of a quiver $\Gamma$, we obtain an oriented surface with boundary $\Sigma_\Gamma$, as illustrated in \eqref{fyqqebnn}. 
More precisely, we associate to each edge a cylinder, and glue these cylinders along the vertices using spheres with disks removed. 
By orienting each boundary component according to the standard orientation for vertices in $\partial\Gamma^+$ and the opposite one for $\partial\Gamma^-$, the resulting surface $\Sigma_\Gamma$ is a two-dimensional cobordism from $|\partial\Gamma^-|$ incoming to $|\partial\Gamma^+|$ outgoing circles.

As we will show in Proposition \ref{pahn7k6y} below, two quivers that determine isomorphic cobordisms are related by a \emph{quiver homotopy}, defined as follows.

\begin{Def}\label{nyfvbnh4}
Two quivers are \defn{homotopic} if one can be obtained from the other by a finite sequence of moves of the form
\begin{equation}\label{aq07x7n5}
\begin{tikzpicture}[
  baseline=(Lpt.center),
  vertex/.style={circle, fill=black, draw=none, inner sep=1.6pt, outer sep=0pt},
  edge/.style={line width=0.9pt, line cap=round}
]
  \coordinate (Lpt) at (0,0);
  \coordinate (Rpt) at (1.5,0);
  \draw[edge] (Lpt) -- (Rpt);
  \foreach \ang in {165,147.5,130,-130,-147.5,-165} {\draw[edge] (Lpt) -- ++(\ang:1.6);}
  \foreach \ang in {15,32.5,50,-50,-32.5,-15} {\draw[edge] (Rpt) -- ++(\ang:1.6);}
  \foreach \yy in {0.15,0.0,-0.15} {
    \fill ($(Lpt)+(-1.4,\yy)$) circle (0.55pt);
    \fill ($(Rpt)+( 1.4,\yy)$) circle (0.55pt);
  }
  \node[vertex] at (Lpt) {};
  \node[vertex] at (Rpt) {};
\end{tikzpicture}
\qquad \longleftrightarrow \qquad
\begin{tikzpicture}[
  baseline=(C.center),
  vertex/.style={circle, fill=black, draw=none, inner sep=1.6pt, outer sep=0pt},
  edge/.style={line width=0.9pt, line cap=round}
]
  \coordinate (C) at (0,0);
  \foreach \ang in {165,147.5,130,-130,-147.5,-165,15,32.5,50,-50,-32.5,-15} {
    \draw[edge] (C) -- ++(\ang:1.6);
  }
  \foreach \yy in {0.15,0.0,-0.15} {
    \fill ($(C)+(-1.4,\yy)$) circle (0.55pt);
    \fill ($(C)+( 1.4,\yy)$) circle (0.55pt);
  }
  \node[vertex] at (C) {};
\end{tikzpicture}
\end{equation}
in either direction, where the number of edges on each side is arbitrary but positive.
\end{Def}

In other words, moving from left to right in \eqref{aq07x7n5} corresponds to deleting an edge between two interior vertices and merging the vertices into one. The reverse move duplicates an interior vertex and creates a new edge between the copies. 
The orientation of the edges in either diagram is arbitrary.

In particular, the orientation of edges not adjacent to the boundary may be changed under homotopy. 
Indeed, one can apply \eqref{aq07x7n5} from left to right, and then again from right to left, where the new edge is the old one but with opposite orientation. 
On the other hand, the orientation of edges adjacent to the boundary cannot be changed. 
Thus the integers
\[
m(\Gamma)\coloneqq|\partial\Gamma^-|,\qquad n(\Gamma)\coloneqq|\partial\Gamma^+|
\]
are homotopy invariants. 
It is also clear that
\[
g(\Gamma)\coloneqq|E|-|V|+1
\]
is a homotopy invariant.

\begin{Prop}\label{pahn7k6y}
Two quivers are homotopic if and only if they induce isomorphic cobordisms. Moreover, the cobordism associated with a connected quiver $\Gamma$ is the oriented surface of genus $g(\Gamma)$ with $m(\Gamma)$ incoming and $n(\Gamma)$ outgoing boundary components.
\end{Prop}

\begin{proof}
It is clear that homotopic quivers induce isomorphic cobordisms since both sides of \eqref{aq07x7n5} represent homeomorphic surfaces. 
For the converse, observe first that the octopus-shaped quiver
\begin{center}
\begin{tikzpicture}[
  baseline=(C.center),
  vertex/.style={circle, fill=black, draw=none, inner sep=1.2pt, outer sep=0pt},
  edge/.style={line width=0.9pt, line cap=round},
  chevW/.store in=\chevW, chevW=0.10,
  chevH/.store in=\chevH, chevH=0.07,
  markchev/.style={
    postaction={
      decorate,
      decoration={
        markings,
        mark=at position #1 with {
          \draw[edge] (-\chevW,-\chevH) -- (0,0) -- (-\chevW,\chevH);
        }
      }
    }
  },
  markchev/.default=0.5
]
  \coordinate (C) at (0,0);
  \def\len{1.6}

  \foreach \ang/\i in {150/1,170/2,-170/3,-150/4} {
    \path (C) -- ++(\ang:\len) coordinate (BL\i);
    \draw[edge] (C) -- (BL\i);
    \path[markchev] (BL\i) -- (C);
    \node[vertex] at (BL\i) {};
  }
  \foreach \ang/\i in {10/1,30/2,-30/3,-10/4} {
    \path (C) -- ++(\ang:\len) coordinate (BR\i);
    \draw[edge] (C) -- (BR\i);
    \path[markchev] (C) -- (BR\i);
    \node[vertex] at (BR\i) {};
  }

  \foreach \h/\w in {3/2.75, 2.5/2.0, 2/1.25} {
    \draw[edge, markchev]
      (C) .. controls +(-\w,\h) and +(\w,\h) .. (C);
  }

  \node[vertex] at (C) {};
\end{tikzpicture}
\end{center}
with $m$ incoming vertices, $n$ outgoing vertices, and $g$ loops induces the surface of genus $g$ with $m$ incoming and $n$ outgoing boundary components.

It therefore suffices to show that every connected quiver is homotopic to a unique octopus. 
Let $\Gamma$ be a connected quiver. 
If $|\Gamma\sint|>1$, then there exists at least one edge not adjacent to $\partial\Gamma$. 
By applying the move \eqref{aq07x7n5} that deletes this edge, we obtain a homotopic quiver with one fewer interior vertex. 
Iterating this process yields a quiver with a single interior vertex, i.e.\ an octopus. 
Uniqueness follows because $g(\Gamma)$, $m(\Gamma)$, and $n(\Gamma)$ are homotopy invariants, and these three quantities completely determine the corresponding octopus by fixing the number of loops and the number of incoming and outgoing legs.
\end{proof}

Our goal is to show that the Lax--Kirchhoff moduli spaces depend only on the homotopy class of a quiver, up to isomorphisms of Hamiltonian spaces. 
In particular, this will allow us to associate to each oriented surface with boundary $\Sigma$ a well-defined Hamiltonian space $\M(\Sigma)$, obtained from any quiver $\Gamma$ whose topological thickening is $\Sigma$.

To establish this, we first describe how the operation of gluing quivers behaves at the level of their moduli spaces. 
This will later imply that gluing cobordisms corresponds to Hamiltonian reduction. 
If $\Gamma_1$ and $\Gamma_2$ are quivers such that $\partial\Gamma_1^+ = \partial\Gamma_2^-$, we define $\Gamma_1 \star \Gamma_2$ to be the quiver obtained by gluing $\partial\Gamma_1^+$ to $\partial\Gamma_2^-$, as illustrated in \eqref{7w8c1eku}.

\begin{Thm}[Gluing]\label{m8uy90np}
Let $\Gamma_1$ and $\Gamma_2$ be connected quivers with $B \coloneqq \partial\Gamma_1^+ = \partial\Gamma_2^- \ne \emptyset$. 
Consider the diagonal Hamiltonian action of $G^B$ on $\M(\Gamma_1) \times \M(\Gamma_2)$. 
Then there is a canonical isomorphism
\[
\M(\Gamma_1 \star \Gamma_2) \cong (\M(\Gamma_1) \times \M(\Gamma_2)) \sll{} G^B
\]
of Hamiltonian $G^{\partial\Gamma_1^-} \times G^{\partial\Gamma_2^+}$-spaces.
\end{Thm}

\begin{proof}
This follows directly from Theorem \ref{tis5zvw5} together with reduction in stages.
\end{proof}

\begin{Thm}[Homotopy invariance]\label{78uo1ham}
Let $\Gamma_1$ and $\Gamma_2$ be two homotopic connected quivers with non-empty boundary $\partial\Gamma \coloneqq \partial\Gamma_1 = \partial\Gamma_2$. 
Then any homotopy between $\Gamma_1$ and $\Gamma_2$ induces a canonical isomorphism $\M(\Gamma_1) \cong \M(\Gamma_2)$ of Hamiltonian $G^{\partial\Gamma}$-spaces.
\end{Thm}

\begin{proof}
Recall that for a Hamiltonian $G \times H$-manifold $M$, there is a canonical isomorphism of Hamiltonian $G \times H$-manifolds
\begin{equation}\label{rvfxrbk7}
(M \times T^*G) \sll{} G \cong M;
\end{equation}
see, for instance, \cite[Theorem 8.18]{MeinrenkenSymplGeom2024}, \cite{MooreTachikawa}, or \cite[Theorem 4.11]{CrooksMayrand2022SymplecticReduction}. 
The isomorphism $\M(\Gamma_1) \cong \M(\Gamma_2)$ is an immediate consequence of this fact combined with Theorem \ref{tis5zvw5}.

More concretely, let $\Gamma$ be a connected quiver with non-empty boundary and $|\Gamma\sint|>1$. 
Let $\Gamma'$ be the quiver obtained from $\Gamma$ by performing the move \eqref{aq07x7n5} from left to right. 
That is, we delete an edge $e_0 \in E$ with $v_1 \coloneqq s(e_0)$ and $v_2 \coloneqq t(e_0)$ both in $\Gamma\sint$, so that 
\[
E' = E \setminus \{e_0\}, \qquad 
V' = V \setminus \{v_2\}, \qquad 
t' = t|_{E'}, \qquad 
s'(e) =
\begin{cases}
v_1 & \text{; if } s(e) = v_2\\
s(e) & \text{; otherwise.}
\end{cases}
\]
It suffices to show that $\M(\Gamma)$ and $\M(\Gamma')$ are isomorphic as Hamiltonian $G^{\partial\Gamma}$-spaces. 
By Theorem \ref{tis5zvw5} and \eqref{rvfxrbk7}, we have
\begin{align*}
\M(\Gamma) 
&\cong T^*G^E \sll{} G^{\Gamma\sint} \\
&\cong (T^*G^{E \setminus \{e_0\}} \times T^*G) \sll{} (G^{\Gamma\sint \setminus \{v_1, v_2\}} \times G^{v_1} \times G^{v_2}) \\
&\cong T^*G^{E \setminus \{e_0\}} \sll{} (G^{\Gamma\sint \setminus \{v_1, v_2\}} \times G^{v_1}) \\
&\cong \M(\Gamma').\qedhere
\end{align*}
\end{proof}

In particular, for every connected oriented two-dimensional cobordism $\Sigma$ with non-empty boundary, we have a Lax--Kirchhoff moduli space $\M(\Sigma)$, well-defined up to isomorphisms of Hamiltonian spaces. If $\Sigma$ has genus $g$ with $m$ incoming and $n$ outgoing boundary components, then
\[
\dim \M(\Sigma) = 2(g + m + n - 1)\dim G.
\]

We have so far restricted to connected quivers for simplicity, but all results hold more generally for a quiver $\Gamma$ all of whose connected components $\Gamma_1, \ldots, \Gamma_k$ have non-empty boundary. In this case,
\[
\M(\Gamma) = \M(\Gamma_1) \times \cdots \times \M(\Gamma_k)
\]
is again a finite-dimensional smooth Hamiltonian $G^{\partial\Gamma}$-space. 
It follows that for every oriented two-dimensional cobordism $\Sigma$ all of whose connected components have non-empty boundary, there is a Lax--Kirchhoff moduli space $\M(\Sigma)$.

\section{Topological quantum field theories}\label{j8sih92i}

Recall that two-dimensional cobordisms form a category $\mathbf{Cob}_2$ whose objects are compact one-dimensional manifolds (disjoint union of circles) and a morphism from $M$ to $N$ is an oriented surface $\Sigma$ whose boundary is $\partial\Sigma = M^- \sqcup N$, where $M^-$ is $M$ with the opposite orientation. We denote $\partial\Sigma^- \coloneqq M$ and $\partial\Sigma^+ \coloneqq N$. Two cobordisms $\Sigma_1 : \partial\Sigma_1^- \to \partial\Sigma_1^+$ and $\Sigma_2 : \partial\Sigma_2^- \to \partial\Sigma_2^+$ such that $\partial\Sigma_1^+ = \partial\Sigma_2^-$ are composed by gluing $\Sigma_1$ with $\Sigma_2$ along $\partial\Sigma_1^+ = \partial\Sigma_2^-$, resulting in a cobordism $\Sigma_2 \circ \Sigma_1 : \partial\Sigma_1^- \to \partial\Sigma_2^+$. Theorem \ref{m8uy90np} then shows that the composition of cobordisms corresponds to Hamiltonian reduction, that is
\begin{equation}\label{66b1kwd2}
\M(\Sigma_2 \circ \Sigma_1) = (\M(\Sigma_1) \times \M(\Sigma_2)) \sll{} G^n,
\end{equation}
where $n$ is the number of connected components of $\partial\Sigma_1^+ = \partial\Sigma_2^-$.

Following \cite{MooreTachikawa}, we now interpret the association $\Sigma \mapsto \M(\Sigma)$ and the gluing law \eqref{66b1kwd2} functorially, where the target is a certain category of Hamiltonian spaces.
To do so, we first define a partial category: that is, a category where only some morphisms can be composed (a paradigm example of such a partial category is Weinstein's symplectic ``category'' of Lagrangian correspondences \cite{Weinstein1982symplecticcategory,Weinstein2010Symplecticcategories}). The objects of our partial category are compact Lie groups and a morphism from $G$ to $H$ is an isomorphism class of Hamiltonian $G \times H$-spaces. Two morphisms $M : G \to H$ and $N : H \to I$ are \emph{composable} if $H$ acts freely on $M \times N$. In this case, the composition is the symplectic reduction
\[
N \circ M \coloneqq (M \times N) \sll{} H : G \too I,
\]
which is indeed a Hamiltonian $G \times I$-space. The identity morphism of a Lie group $G$ is the cotangent bundle $T^*G$ with its Hamiltonian $G \times G$-action, as follows from \eqref{rvfxrbk7}.
By the Wehrheim--Woodward construction \cite{WehrheimWoodward2010Functoriality}, such a partial category can be canonically completed to a category, denoted $\mathbf{Ham}$.
The objects of $\mathbf{Ham}$ are identical (compact Lie groups) but the morphisms are now finite sequences of the original morphisms, modulo the equivalence relation generated by composing adjacent composable pairs (see also \cite{CrooksMayrand2024MooreTachikawa,Cazassus2023twocategory}).
Then $\mathbf{Ham}$ is a symmetric monoidal category under the cartesian product of Lie groups and Hamiltonian manifolds. 
The fact that the monoidal structure is well-defined can be proved exactly as in \cite[\S3.3]{CrooksMayrand2024MooreTachikawa} (or alternatively by embedding $\mathbf{Ham}$ into the 1-shifted symplectic category $\mathbf{WS}_1$ as in \cite[\S5]{CrooksMayrand2024MooreTachikawa}). There is also a 2-category upgrade of $\mathbf{Ham}$ considered in \cite{Cazassus2023twocategory}.

Our goal is to show that the Lax--Kirchhoff moduli spaces induce a unique symmetric monoidal functor
\[
\M : \mathbf{Cob}_2 \too \mathbf{Ham},
\]
sending the circle $S^1$ to $G$ and $\Sigma_\Gamma$ to $\M(\Gamma)$ for every $\Gamma$, i.e.\ a two-dimensional topological quantum field theory (TQFT) valued in $\mathbf{Ham}$. See, for example, \cite{Kock2004Frobenius} for background on TQFTs valued in arbitrary symmetric monoidal categories.

To do so, it suffices to construct the functor on a finite number of generators of $\mathbf{Cob}_2$ subject to a finite number of relations, as listed in \cite[\S1.4]{Kock2004Frobenius}. In more detail, denoting by $\Sigma_{m, n}$ the genus-0 cobordism from $m$ circles to $n$ circles, the generators are the cup $\Sigma_{1, 0} = \tqftcup$, the cap $\Sigma_{0, 1} = \tqftcap$, the two pairs of pants $\Sigma_{2, 1} = \tqftpoptwoone$ and $\Sigma_{1, 2} = \tqftpoponetwo$, the cylinder $\Sigma_{1, 1} = \tqftcyl$ and the swap $\tqftswap$. Among this list, only the cup $\tqftcup$ and the cap $\tqftcap$ have not been associated to a morphism in $\mathbf{Ham}$, as they are not induced by any quiver. We extend the definition by setting $\M(\tqftcup)$ and $\M(\tqftcap)$ to be a singleton $\{*\}$ equipped with the trivial $G$-action. As we will see below, this definition is forced by functoriality. The relations that we need to verify so that this generates a symmetric monoidal functor are all consequences of homotopy invariance (Theorem \ref{78uo1ham}) and gluing (Theorem \ref{m8uy90np}) except for those involving the cup and cap, since the latter are not induced by quivers.
That is, it remains to verify that capping off a boundary component, i.e.\ removing an outgoing leg from the quiver, corresponds to composition with $\M(\tqftcup)$ in $\mathbf{Ham}$, and similarly for an incoming leg.
Since composing $\M(\Sigma_\Gamma)$ with $\M(\tqftcup) = \{*\}$ amounts to symplectic reduction by the $G$-action of the corresponding boundary component, the required compatibility is the content of the next proposition.

\begin{Prop}\label{bij94btr}
Let $\Gamma$ be a connected quiver with non-empty boundary and let $v_0 \in \partial\Gamma$. 
Let $G^{v_0}$ denote the copy of $G$ in $G^{\partial\Gamma}$ associated with $v_0$, and consider the induced Hamiltonian action of $G^{v_0}$ on $\M(\Gamma)$. 
Then
\[
\M(\Gamma) \sll{} G^{v_0} \cong \M(\Gamma \setminus \{v_0\}),
\]
where $\Gamma \setminus \{v_0\}$ is obtained from $\Gamma$ by deleting $v_0$ and the unique edge adjacent to $v_0$.
\end{Prop}

\begin{proof}
This follows again from Theorem \ref{tis5zvw5}. 
Suppose that $v_0 \in \partial\Gamma^+$ (the case $\partial\Gamma^- \ne\emptyset$ is similar), and let $e_0$ be the unique edge such that $t(e_0) = v_0$.
Then
\begin{align*}
\M(\Gamma) 
&= T^*G^E \sll{} G^{\Gamma\sint} \\
&= (T^*G^{E \setminus \{e_0\}} \times T^*G^{e_0}) \sll{} G^{\Gamma\sint},
\end{align*}
where the Hamiltonian $G^{v_0}$-action is induced by the left action on $T^*G^{e_0}$. 
Since the Hamiltonian reduction of $T^*G$ by the left action of $G$ is trivial, we obtain
\[
\M(\Gamma) \sll{} G^{v_0} 
= (T^*G^{E \setminus \{e_0\}} \times \{*\}) \sll{} G^{\Gamma\sint} 
= \M(\Gamma \setminus \{v_0\}).\qedhere
\]
\end{proof}

We therefore have (by e.g.\ \cite[Theorem 3.6.19]{Kock2004Frobenius}) a symmetric monoidal functor $\mathbf{Cob}_2 \to \mathbf{Ham}$, i.e.\ a two-dimensional topological quantum field theory valued in $\mathbf{Ham}$. 

We note that only the closed oriented surfaces without boundary are sent to non-trivial sequences of Hamiltonian spaces in the completed category $\mathbf{Ham}$, while all the other ones are smooth (length-$1$) morphisms. These abstract morphisms can nevertheless be represented by stratified symplectic spaces \cite{SjamaarLerman1991Stratified} after performing the singular reduction.

Finally, we remark that the TQFT is uniquely determined by the quiver thickenings $\Sigma_\Gamma$ since this forces $\M(\tqftcap) = \M(\tqftcup) = \{*\}$. Indeed, if $\Gamma$ is any connected quiver with $v_0 \in \partial\Gamma^+$, we must have that $(\M(\Gamma) \times \M(\tqftcup)) \sll{} G \cong \M(\Gamma \setminus \{v_0\})$. By a dimension count, this forces $\dim \M(\tqftcup) = 0$ and connectedness follows from that of $\M(\Gamma)$ and $\M(\Gamma \setminus \{v_0\})$. We have therefore reached the final conclusion of this paper.

\begin{Thm}[TQFT valued in Hamiltonian spaces]
Let $G$ be a compact connected Lie group. There is a unique two-dimensional topological quantum field theory
\[
\mathbf{Cob}_2 \too \mathbf{Ham}
\]
sending the circle to $G$ and the thickening $\Sigma_\Gamma$ of a connected quiver $\Gamma$ with non-empty boundary to the Lax--Kirchhoff moduli space $\M(\Gamma)$.\qed
\end{Thm}

\bibliographystyle{plain}
\bibliography{lax-kirchhoff-moduli.bib}

\end{document}